\documentclass[12pt]{article}
\usepackage{amsfonts, amssymb, amsmath}
\usepackage{amsthm}
\usepackage[all]{xy}
\usepackage{url}
\usepackage{graphicx}
\usepackage{listings}
\newtheorem{prop}{Proposition}[section]
\newtheorem{theorem}{Theorem}[section]
\newtheorem{lemma}{Lemma}[section]

\newcommand{\Lm}{\mathcal{L}^M}

\newcommand{\lcm}{\textrm{lcm}}

\newcommand{\N}{\mathbb{N}}
\newcommand{\Z}{\mathbb{Z}}
\newcommand{\Q}{\mathbb{Q}}
\newcommand{\R}{\mathbb{R}}
\newcommand{\C}{\mathbb{C}}

\begin{document}
\begin{center}
\textbf{\Large On the reciprocal sum of lcm of $k$-tuples} \\
\vspace{5mm}
\begin{tabular}{c} Sungjin Kim\\[1.2\baselineskip]
\multicolumn{1}{c}{Santa Monica College, California State University Northridge} \\
 \verb+sungjin.kim@csun.edu+
\end{tabular} \end{center}

\begin{abstract}
We prove that the reciprocal sum $S_k(x)$ of the least common multiple of $k\geq 3$ positive integers in $\N\cap [1,x]$ satisfies
$$
S_k(x)=P_{2^k-1}(\log x)+O(x^{-\theta_k+\epsilon})
$$
where $P$ is a polynomial of degree $2^k-1$ and $\theta_k=\frac{2^k}{(k+1)^{\frac{k+1}2}}\cdot \frac3{2^k+6k-5}$. This was conjectured in Hilberdink, Luca, and T\'{o}th~\cite[Remark 2.4]{HLT}. We also prove asymptotic formulas for similar sums conjectured there.
\end{abstract}

\lstset{
language=python,
basicstyle= \ttfamily,
captionpos=b,
frame=tb,
columns=fullflexible,
showstringspaces=false
}
\section{Introduction}
    Let $x\geq 2$ be an integer, $\N=\{1,2,3,\ldots\}$ be the set of positive integers, and $\mathcal{P}$ be the set of prime numbers. Let $(n_1,\ldots,n_k)=\gcd(n_1,\ldots, n_k)$ and $[n_1,\ldots, n_k]=\lcm(n_1,\ldots,n_k)$ be the greatest common divisor (gcd) and the least common multiple (lcm) of $k$-tuple of positive integers $n_1, \ldots, n_k$ respectively. If $(n_1,\ldots, n_k)$ means the $k$-tuple of positive integers, then they will be written as $(n_1,\ldots, n_k)\in \N^k$. In this paper, we denote $P_v(x)$, $P^{(1)}_v(x)$, $P^{(2)}_v(x)$, $\ldots$ , polynomials of degree at most $v$ in variable $x$. They are not necessarily the same at each occurrence.

    The reciprocal sum of gcd is well understood through Dirichlet convolution of arithmetic functions. We have~\cite[Page 3]{HLT} for $k\geq 3$,
    $$
    \sum_{n_1,\ldots, n_k\leq x} \frac1{(n_1,\ldots, n_k)}=\frac{\zeta(k+1)}{\zeta(k)} x^k + O(x^{k-1}),
    $$
    and by~\cite[Theorem 1.1]{L}, we have for $k=2$,
    $$
    \sum_{n_1, n_2 \leq x} \frac1{(n_1,n_2)}=\frac{\zeta(3)}{\zeta(2)}x^2+O\left(x(\log x)^{2/3}(\log\log x)^{1/3}\right).
    $$
    The reciprocal sums involving the lcm are more delicate. The $r$-th moment results were studied by Hilberdink and T\'{o}th~\cite{HT} by means of multivariable Dirichlet series. They proved that for $r>-1$,
    \begin{equation}
    \sum_{n_1,\ldots, n_k\leq x} [n_1,\ldots, n_k]^r= \frac{C_{r,k}}{(r+1)^k} x^{k(r+1)} + O\left( x^{k(r+1)-\frac12 \textrm{min}(r+1,1)+\epsilon}\right)
    \end{equation}
    and
    \begin{equation}
    \sum_{n_1,\ldots, n_k\leq x} \left(\frac{[n_1,\ldots, n_k]}{n_1\cdots n_k} \right)^r  = C_{r,k} x^k + O\left( x^{k-\frac12\textrm{min}(r+1,1)+\epsilon}\right)
    \end{equation}
    where $C_{r,k}$ is given as an Euler product.

    The results in case $r=-1$ in (1) and (2) were not obtained in~\cite{HT}. Hilberdink, Luca, and T\'{o}th~\cite{HLT} studied these sums and proposed some open problems. They proved that as $x\rightarrow\infty$,
    \begin{equation}
    S_k(x):=\sum_{n_1,\ldots, n_k\leq x} \frac1{[n_1,\ldots, n_k]} \asymp \log^{2^k -1} x,
    \end{equation}
    \begin{equation}
    U_k(x):=\sum_{\substack{{n_1,\ldots, n_k\leq x}\\{(n_1,\ldots, n_k)=1}}} \frac1{[n_1,\ldots, n_k]}\asymp \log^{2^k-2} x,\end{equation}
    and
    \begin{equation}
    x^k\ll V_k(x):=\sum_{n_1,\ldots, n_k\leq x} \frac{n_1\cdots n_k}{[n_1,\ldots, n_k]}\ll x^k(\log x)^{2^k-2}.
    \end{equation}
    Then they conjectured that
    $$
    S_k(x)=P_{2^k-1}(\log x) + O(x^{-r})
    $$
    and
    $$
    V_k(x)\sim \lambda_k x^k (\log x)^{2^k-k-1},
    $$
    where $P_{2^k-1}$ is a polynomial of degree $2^k-1$ with a positive leading coefficient, $r>0$, and $\lambda_k>0$.

    Recently, the author~\cite[Theorem 1.2, and 1.3]{K} improved the error terms of (1) and (2) and provided the correct order of magnitude of (5). In fact,
    $$
    \sum_{n_1,\ldots, n_k\leq x} [n_1,\ldots, n_k]^r= \frac{C_{r,k}}{(r+1)^k} x^{k(r+1)} + O\left( x^{k(r+1)} E_{r,k}(x)\right)
    $$
    and
    $$
    \sum_{n_1,\ldots, n_k\leq x} \left(\frac{[n_1,\ldots, n_k]}{n_1\cdots n_k} \right)^r  = C_{r,k} x^k + O\left( x^k E_{r,k}(x)\right)
    $$
    where
    $$
    E_{r,k}(x)=\begin{cases}O_r(x^{-\frac{r+1}2} \log^{\max(2^k-k-1, 2k^2-k-2)} x) \mbox{ if } -1<r\leq 1, \\
    O_r(x^{-1}\log^{k-1} x) \mbox{ if } r>1.\end{cases}.
    $$
    Moreover,
    $$
    V_k(x)\asymp x^k (\log x)^{2^k-k-1}.
    $$
    The method in this recent work was based on the study of $G$-wise coprime tuples by T\'{o}th~\cite{T} and Hu~\cite{H1},~\cite{H2}. The case when $G=K_k$ is a complete graph was sufficient.

    L. T\'{o}th informed the author that D. Essouabri, C. Salinas Zavala, L. T\'{o}th~\cite[Corollary 2, 3, and 4]{EST} proved among other things that there exist $\theta>0$ such that
    $$
    S_k(x)=\sum_{n_1,\ldots, n_k\leq x} \frac1{[n_1,\ldots, n_k]} = P^{(1)}_{2^k-1}(\log x) + O(x^{-\theta}),
    $$
    $$
    U_k(x)=\sum_{\substack{{n_1,\ldots, n_k\leq x}\\{(n_1,\ldots, n_k)=1}}} \frac1{[n_1,\ldots, n_k]}=P^{(2)}_{2^k-2}(\log x)+O(x^{-\theta}),
    $$
    and
    $$
    V_k(x)=\sum_{n_1,\ldots, n_k\leq x} \frac{n_1\cdots n_k}{[n_1,\ldots, n_k]} = x^k P^{(3)}_{2^k-k-1}(\log x)+O(x^{k-\theta}).
    $$
    The leading coefficients of $P_{2^k-1}^{(1)}$, $P_{2^k-2}^{(2)}$, and $P_{2^k-k-1}^{(3)}$ are explicitly determined in~\cite[Corollary 2, 3, and 4]{EST}.

    In this paper, we prove explicit power-savings in these estimates. The leading coefficients of $P_{2^k-1}^{(1)}$, $P_{2^k-2}^{(2)}$, and $P_{2^k-k-1}^{(3)}$ are determined in Section 3 and 4 independently of~\cite{EST}.
    \begin{theorem}
    As $x\rightarrow\infty$, we have for $k\geq 3$,
    \begin{equation}S_k(x)= P^{(1)}_{2^k-1}(\log x) + O(x^{-\theta^{(1)}_k+\epsilon}),\end{equation}
    \begin{equation}U_k(x)=P^{(2)}_{2^k-2}(\log x)+O(x^{-\theta^{(2)}_k+\epsilon}),\end{equation}
    \begin{equation}V_k(x)= x^k P^{(3)}_{2^k-k-1}(\log x)+O(x^{k-\theta^{(3)}_k+\epsilon})\end{equation}
    where the leading coefficient of $P_{2^k-1}^{(1)}$ is $c_k^{(1)}=c_k>0$, the leading coefficient $c_k^{(2)}$ of $P_{2^k-2}^{(2)}$ satisfies $c_k^{(2)}=(2^k-1)c_k$. For each $k\geq 2$, $c_k>c_{k+1}$ and $\lim_{k\rightarrow\infty} c_k=0$. Moreover, for any $1\leq i<j\leq 3$, we have $c_k^{(i)}/c_k^{(j)}\in\Q\cap (0,\infty)$. We have for $k\geq 3$,
    $$
    \theta^{(1)}_k=\theta^{(3)}_k=\frac{2^k}{(k+1)^{\frac{k+1}2}}\cdot \frac 3{2^k+6k-5}, \ \ \theta^{(2)}_k=\frac{2^k}{(k+1)^{\frac{k+1}2}}\cdot \frac 3{2^k+6k-6}.
    $$
    \end{theorem}

    Our method relies on the non-inductive argument in the counting of $G$-wise coprime tuples by Arias-de-Reyna and Heyman~\cite{RH}. By their argument, we are able to write $S_k(x)$ as a sum over $(2^k-1)$-tuples of positive integers. We prove a general result (Theorem 2.2) on the reciprocal sum over $G$-wise coprime tuples in Section 2 by an elementary method, then we obtain the leading terms of $S_k(x)$, $U_k(x)$, and $V_k(x)$ as a corollary in Section 3. Further in Section 3, we apply a version of multivariable Perron's formula to strengthen Theorem 2.2 into a full asymptotic formula with a power-saving error term. Then we complete the proof of Theorem 1.1 as a corollary. The computations of $c_k$ are obtained in Section 4. All asymptotic formulas given in this paper are for $x\rightarrow\infty$. The notations $A(x)\ll B(x)$ or $A(x)=O(B(x))$ mean that $|A(x)/B(x)|$ is bounded as $x\rightarrow\infty$. For simplicity of exposition, the positive numbers $M$ without subscripts and $\epsilon$ may appear multiple times, but they are not necessarily the same everytime. We use $\mathcal{L}=\log x$ and $\Lm$ for any fixed power of $\log x$ so that $O(\Lm)\cdot O(\Lm) = O(\Lm)$, any divisor function $\tau_v(n)=\sum_{a_1\cdots a_v=n}1$ satisfies $\tau_v(n)\ll n^{\epsilon}$, $\Lm\ll x^{\epsilon}$, and $O(x^{\epsilon}) \cdot O(x^{\epsilon}) = O(x^{\epsilon})$.  \\

    {\bf Data Availability Statement.} All data generated or analysed during this study are included in this published article and its supplementary information files.

    {\bf Acknowledgments.} The author thanks Titus Hilberdink for initially mentioning the multivariable Perron's formula. The author also thanks Sary Drappeau for bringing~\cite[Theorem 1]{dlB} to his attention and helpful conversation about the paper.
    \section{$G$-wise coprime $k$-tuples}
    Let $x\geq 2$ be an integer and $G=(V,E)$ be any graph with $v\geq 2$ vertices and $e\geq 1$ edges. Let $V=\{1,\ldots, v\}$ and $E=\{\epsilon_1,\ldots, \epsilon_e\}\subseteq \{(i,j) \ | \ 1\leq i<j\leq v\}$ so that we have labels on vertices and edges of $G$. We are interested in a reciprocal sum over $v$-tuples $(a_1,\ldots, a_v)\in (\N\cap [1,x])^v$ such that $(a_i,a_j)=1$ whenever $(i,j)\in E$. Such $v$-tuples are called {\it $G$-wise coprime} in~\cite{H1},~\cite{H2}, and~\cite{RH}. The asymptotic formula of the counting function of $G$-wise coprime tuples is obtained in~\cite{H2} by an inductive argument and~\cite{RH} by an non-inductive argument. For each vertex $r$, let $M_r$ be the least common multiple of $m_i$ where the edge $\epsilon_i$ connects the vertex $r$ to another vertex of $G$. The number of $G$-wise coprime tuples given in~\cite{RH} is
    $$
    \sum_{\substack{{a_1,\ldots, a_v\leq x}\\{G\textrm{-wise coprime}}}}1=\sum_{m_1,\ldots, m_e\leq x}\mu(m_1)\cdots \mu(m_e) \prod_{r=1}^v \left\lfloor \frac x{M_r} \right\rfloor.
    $$
    They obtained this formula by inclusion-exclusion principle. We provide a more direct approach to prove this. Recall that when we sum over the coprime pairs $(a,b)$, we insert $\sum_{d|a, d|b} \mu(d)$ into the summation. To each edge $\epsilon_i=(a,b)\in E$, we insert $\mu(m_i)$ and the divisibility conditions $m_i|a$, $m_i|b$. Combining all divisibility conditions of each vertex $r$, we obtain the result. Similarly for reciprocal sums, we have
    \begin{equation}
    \sum_{\substack{{a_1,\ldots, a_v\leq x}\\{G\textrm{-wise coprime}}}} \frac1{a_1\cdots a_v} = \sum_{m_1,\ldots, m_e\leq x} \mu(m_1)\cdots \mu(m_e)\sum_{\substack{{a_1,\ldots, a_v \leq x}\\{\forall_r, M_r | a_r}}} \frac1{a_1\cdots a_v}.
    \end{equation}
    The asymptotic formula for the reciprocal sums over $G$-wise coprime tuples is as follows.
    \begin{theorem}
    Let $d$ be the maximal degree of the vertices of $G$. We have
    $$
    \sum_{\substack{{a_1,\ldots, a_v\leq x}\\{G\textrm{-wise coprime}}}} \frac1{a_1\cdots a_v} =P_v(\log x) + O(x^{-1}\log^{v+d-1} x)
    $$
    where $P_v$ is a polynomial of degree $v$ with a leading coefficient $\rho(G)>0$.
    \end{theorem}
    \begin{proof}
    The inner sum of (9) is
    $$
    \frac1{M_1\cdots M_v}\sum_{\forall_r, b_r\leq \frac x{M_r}} \frac1{b_1\cdots b_v}=\frac1{M_1\cdots M_v}\prod_{r=1}^v \left( \log\frac x{M_r}+\gamma+O\left(\frac{M_r}x\right)\right).
    $$
    It was proved in~\cite{RH} that for any $1\leq r\leq v$ with obvious modification for $r=1, v$,
    $$
    \sum_{m_1,\ldots,m_e\leq x} |\mu(m_1)\cdots \mu(m_e)|\frac1{M_1\cdots M_{r-1}M_{r+1} \cdots M_v}=O(\log^d x)
    $$
    Thus, the $O$-terms in the product contribute $O(x^{-1}\log^{v-1+d}x)$.

    The coefficients of $P_v$ are linear combinations of the absolutely convergent sums
    $$
    \sum_{m_1,\ldots, m_e=1}^{\infty} \mu(m_1)\cdots \mu(m_e) \frac{(\gamma-\log M_{r_1})\cdots (\gamma-\log M_{r_s})}{M_1\cdots M_v}
    $$
    and the sums over $m_i>x$ for some $i\leq e$ in the above contribute to the error term. In fact, it was proved in~\cite{RH} that if $m_i$ is the largest among $m_1,\ldots, m_e$, then
    \begin{equation}
    \sum_{m_1,\ldots, m_{i-1}\leq m_i} \sum_{m_i>x} \sum_{m_{i+1},\ldots, m_e\leq m_i} |\mu(m_1)\cdots \mu(m_e)|\frac1{M_1\cdots M_v} = O\left( \frac{(\log\log x)^w}x\right)
    \end{equation}
    for some $w=w(G)>0$. Applying the absolute convergence of the Dirichlet series of $\zeta(s)^{-1}\sum \frac{|\mu(n)|}{n^s}\left(\frac{\sigma(n)}n\right)^w$ over $\Re (s) > 1/2$, we have
    $$
    \sum_{n>x} \frac{|\mu(n)|}{n^2}\left(\frac{\sigma(n)}n\right)^w = O\left(\frac1x\right).
    $$
    Thus, the error term of (10) can be improved to $O\left( \frac1x\right)$. This yields
    \begin{align*}
    \sum_{m_1,\ldots, m_{i-1}\leq m_i} &\sum_{m_i>x} \sum_{m_{i+1},\ldots, m_e\leq m_i} |\mu(m_1)\cdots \mu(m_e)|\frac{(\gamma-\log M_{r_1})\cdots (\gamma-\log M_{r_s})}{M_1\cdots M_v}= O\left( \frac{\log^s x}x\right).
    \end{align*}
    The error term from the coefficients is therefore $O(x^{-1} \log^{s+v-s} x)=O(x^{-1} \log^v x)$.
    \end{proof}
    The leading coefficient of $P_v$ ($s=0$ above) is
    $$
    \sum_{m_1,\ldots, m_e=1}^{\infty} \mu(m_1)\cdots \mu(m_e) \frac1{M_1\cdots M_v} = \rho(G)
    $$
    that will be further discussed in Section 3 and Section 4.

    Now we consider the reciprocal sum over $G$-wise coprime tuples with some {\it hyperbolic constraints} $\prod_{j\in A_i} a_j \leq x$ for some subsets $A_i\subseteq V$, $i=1,\ldots k$ such that $\cup A_i = V$. As noted in~\cite{RH}, the sums over $m_i$ can be extended to all tuples of positive integers. If $m_i>x$ for some $m_i$, then the inner sum over $a_j$ with $j\in\epsilon_i$ vanishes. Thus, we are considering the sum
    \begin{align*}
    \sum_{\substack{{a_1,\ldots, a_v\leq x}\\{G\textrm{-wise coprime}}\\{\forall_{i\leq k}, \ \prod_{j\in A_i}a_j \leq x}}}& \frac1{a_1\cdots a_v}= \sum_{m_1,\ldots, m_e=1}^{\infty} \mu(m_1)\cdots \mu(m_e)\sum_{\substack{{a_1,\ldots, a_v \leq x}\\{\forall_r, M_r | a_r}\\{\forall_{i\leq k}, \ \prod_{j\in A_i}a_j \leq x}}} \frac1{a_1\cdots a_v}\\
    &=\sum_{m_1,\ldots, m_e=1}^{\infty} \mu(m_1)\cdots \mu(m_e)\frac1{M_1\cdots M_v}\sum_{\substack{{\forall_r, b_r\leq \frac x{M_r}}\\{\forall_{i\leq k}, \ \prod_{j\in A_i}b_j \leq x/\prod_{j\in A_i} M_j}}} \frac1{b_1\cdots b_v}\\
    &=\sum_{m_1,\ldots, m_e=1}^{\infty} \mu(m_1)\cdots \mu(m_e)\frac1{M_1\cdots M_v}\sum_{\forall_{i\leq k}, \ \prod_{j\in A_i}b_j \leq x/\prod_{j\in A_i} M_j}\frac1{b_1\cdots b_v}.
    \end{align*}
    In the last sum, the conditions $\forall_r, b_r\leq x/M_r$ are dropped since the hyperbolic constraints imply them. The inner sum is more difficult than (9) due to the hyperbolic constraints. However, if we focus on the main term, then the following result is obtained by an elementary method.
    \begin{theorem}
    We have
    $$
    \sum_{\substack{{a_1,\ldots, a_v\leq x}\\{G\textrm{-wise coprime}}\\{\forall_{i\leq k}, \ \prod_{j\in A_i}a_j \leq x}}} \frac1{a_1\cdots a_v}=\rho(G)\mathrm{vol}(D) \log^v x + O(\log^{v-1}x)
    $$
    where $D$ is a convex polytope defined by the hyperbolic constraints $\{A_i\}_{i\leq k}$,
    $$
    D=\left\{(t_1,\ldots, t_v)\in [0,\infty)^v \ \bigg\vert  \forall_{i\leq k}, \ \sum_{j\in A_i}t_j \leq 1 \right\}.
    $$
    \end{theorem}
    \begin{proof}
    We begin with
    \begin{equation}
    \sum_{\substack{{a_1,\ldots, a_v\leq x}\\{G\textrm{-wise coprime}}\\{\forall_{i\leq k}, \ \prod_{j\in A_i}a_j \leq x}}} \frac1{a_1\cdots a_v}= \sum_{m_1,\ldots, m_e=1}^{\infty} \frac{\mu(m_1)\cdots \mu(m_e)}{M_1\cdots M_v}\sum_{\forall_{i\leq k}, \ \prod_{j\in A_i}b_j \leq x/\prod_{j\in A_i} M_j}\frac1{b_1\cdots b_v}.
    \end{equation}
    Recall that for any $1\leq t$, we have
    $$
    \sum_{b\leq t} \frac{1}b = \log t+ O\left(1\right)=\int_1^t \frac{1}b db+O\left(1\right).
    $$
    Using this for the rightmost variable $b_v$, we replace the summation by the integral over corresponding restrictions, then interchange the integral to the leftmost and repeat until we replace all summations by the integrals. At each step of this repetition, the error term $O(1)$ contributes to $O(\log^{v-1}x)$. The inner sum becomes
    $$
    \int_{\substack{{\forall_r, 1\leq b_r}\\{\forall_{i\leq k}, \ \prod_{j\in A_i}b_j \leq x/\prod_{j\in A_i} M_j}}} \frac1{b_1\cdots b_v} db_1\cdots db_v+O(\log^{v-1}x).
    $$
    Applying the change of variables $\log b_r=t_r$ and $t_r/\log x=u_r$, the integral becomes
    \begin{align*}
    \int_{\substack{{\forall_r, 0\leq t_r}\\{\forall_{i\leq k}, \sum_{j\in A_i} t_j \leq \ \log x - \sum_{j\in A_i} \log M_j} }} dt_1\cdots dt_v&=(\log^v x)\int_{\substack{{\forall_r, 0\leq u_r}\\{\forall_{i\leq k}, \sum_{j\in A_i} u_j \leq \ 1 - \frac{\sum_{j\in A_i} \log M_j}{\log x}} }}du_1\cdots du_v\\
    &=(\log^v x)\left( \mathrm{vol} (D)+ O\left( \frac{\sum_{j\leq v} \log M_j}{\log x}\right)\right) \\
    \end{align*}
    Inserting this into the sums over $m_i$, we obtain the result.
    \end{proof}
    \section{Coprimality Graph and Proof of Theorem 1.1}
    For any graph $G=(V,E)$ with $|V|=v$, $|E|=e$, it is proved in~\cite[Proposition 2]{RH} that the following functions $f_G:\N\rightarrow\Z$ and $f_G^+:\N\rightarrow\Z$ are multiplicative where
    $$
    f_G(m)=\sum_{M_1\cdots M_v=m} \mu(m_1)\cdots \mu(m_e), \ \ f_G^+(m)=\sum_{M_1\cdots M_v=m} |\mu(m_1)\cdots \mu(m_e)|.
    $$
    Also,~\cite{RH} shows that $f_G(1)=f_G^+(1)=1$, $f_G(p)=f_G^+(p)=0$, $f_G(p^2)=-e$, $f_G^+(p^2)=e$, and $f_G(p^{\alpha})=f_G^+(p^{\alpha})=0$, for all $\alpha>v$. Thus, we obtain the absolutely convergent Dirichlet series
    $$
    \sum_{m=1}^{\infty} \frac{f_G(m)}{m^s} = \prod_{p\in\mathcal{P}} \left( 1+ \frac{f_G(p^2)}{p^{2s}} + \cdots + \frac{f_G(p^v)}{p^{vs}}\right),  \ \ \Re(s)>\frac12
    $$
    and
    $$
    \sum_{m=1}^{\infty} \frac{f_G^+(m)}{m^s} = \prod_{p\in\mathcal{P}} \left( 1+ \frac{f_G^+(p^2)}{p^{2s}} + \cdots + \frac{f_G^+(p^v)}{p^{vs}}\right),  \ \ \Re(s)>\frac12
    $$
    By~\cite[Lemma 3]{RH}, $f_G(p^a)$ depends only on $G$ and $a$, not on $p$.  Then we have $w=w(G)>0$ such that
    $$
    \sum_{m\leq x} \frac{|f_G(m)|}{\sqrt m}\leq \sum_{m\leq x} \frac{f_G^+(m)}{\sqrt m}\leq \prod_{p\leq x} \left(1+\frac wp\right)\leq C \log^w x.
    $$
    Thus, the series $\sum_{m=1}^{\infty} \frac{f_G(m)}m$ is absolutely convergent. Then the number $\rho(G)$ defined in~\cite{RH} given by the series
    $$
    \rho(G)=\sum_{m_1,\ldots, m_e=1}^{\infty} \frac{\mu(m_1)\cdots \mu(m_e)}{M_1\cdots M_v}=\sum_{m=1}^{\infty} \frac{f_G(m)}m
    $$
    is absolutely convergent.

    For each $F\subseteq E$, let $v(F)$ be the number of non-isolated vertices of $F$. By~\cite[Lemma 3]{RH}, we have
    $$
    \rho(G)=\prod_{p\in\mathcal{P}} Q_G\left(\frac1p\right)
    $$
    where $Q_G(x)$ is a polynomial defined by
    $$
    Q_G(x)=\sum_{F\subseteq E} (-1)^{|F|}x^{v(F)}=1+c_2 x^2 + \cdots + c_v x^v.
    $$
    For each prime $p$ and $a\geq 2$, we have $c_a=f_G(p^a)$. Moreover, $c_2=-e$ by~\cite[Section 2]{RH}.

    By~\cite[Theorem 1]{H2}, we have another expression for $\rho(G)$.
    $$
    \rho(G)=\prod_{p\in\mathcal{P}} \left(\sum_{m=0}^v i_m(G)\left(1-\frac1p\right)^{v-m}\frac1{p^m}\right)
    $$
    where $i_m(G)$ is the number of $S\subseteq V$ with $|S|=m$ such that no two vertices are connected by an edge in $G$ (such a set is called {\it independent set}). In view of this expression, $\rho(G)>0$ for any graph with $v\geq 1$. Analyzing the proofs of these papers, we have the following identity
    $$Q_G(x)=\sum_{F\subseteq E} (-1)^{|F|}x^{v(F)}=\sum_{m=0}^v i_m(G)\left(1-x\right)^{v-m}x^m$$
    for $x=1/p$ for each prime $p$. This gives a proof that the above holds for any $x$. We may have an interesting combinatorial proof to the above identity. However, this is not the main topic of our result. We do not give the combinatorial proof here.
    \subsection{The leading terms of $S_k$ for $k=2,3$}
    For $k=2$, recall by putting $a_3=(n_1,n_2)$, $n_1=a_1a_3$, and $n_2=a_2a_3$ that
    $$
    \sum_{n_1, n_2\leq x} \frac1{[n_1,n_2]}=\sum_{\substack{{a_1,a_2,a_3\leq x}\\{G\textrm{-wise coprime}}\\{a_1a_3, a_2a_3\leq x}}}\frac1{a_1a_2a_3}=\sum_{e\leq x}\mu(e)\sum_{\substack{{a_1,a_2,a_3\leq x}\\{e|a_1, e|a_2}\\{a_1a_3, a_2a_3\leq x}}}\frac1{a_1a_2a_3}
    $$
    where $G$ is the graph $G=(V,E)$ with $V=\{1,2,3\}$ and $E=\{(1,2)\}$. It is easy to see that $\rho(G)=\frac6{\pi^2}$ and $\textrm{vol}(D)=\frac13$ where $D$ is $\{(x,y,z)\in [0,1]^3 | x+y\leq 1, \ x+z\leq 1\}$.

    We give a detailed explanation of how the inner sum (which is a triple sum in case $k=2$) in (11) is replaced by the triple integral at a cost of an error term. Let $A, B\geq 1$. In the following, $u\wedge v = \min(u,v)$ and we use indices $a$, $b$, and $c$ for simplicity of notations. We have
    \begin{align*}
    \sum_{\substack{{ac\leq x/A}\\{bc\leq x/B}}} \frac1{abc}&=\sum_{a\leq \frac xA}\sum_{b\leq \frac xB}\sum_{c\leq \frac x{Aa} \wedge \frac x{Bb}} \frac1{abc}\\
    &=\sum_{a\leq \frac xA} \sum_{b\leq \frac xB} \frac1{ab} \left(\int_1^{\frac x{Aa} \wedge \frac x{Bb}} \frac1c \ dc + O(1)\right)\\
    &=\int_1^{\frac x{A} \wedge \frac x{B}} \frac 1c \sum_{a\leq \frac x{Ac}} \frac 1a \sum_{b\leq \frac x{Bc}} \frac1b  \ dc + O(\log^2 x)\\
    &=\int_1^{\frac x{A} \wedge \frac x{B}} \frac 1c \sum_{a\leq \frac x{Ac}} \frac 1a \left( \int_1^{\frac x{Bc}} \frac1b \ db +O(1) \right) dc + O(\log^2 x)\\
    &=\int_1^{\frac xB} \frac1b \int_1^{\frac xA \wedge \frac x{Bb}} \frac1c \sum_{a\leq \frac x{Ac}} \frac 1a  \ dc db+ O(\log^2 x)\\
    &=\int_1^{\frac xB} \frac1b \int_1^{\frac xA \wedge \frac x{Bb}} \frac1{c}  \left( \int_1^{\frac x{Ac}} \frac 1a \ da + O(1) \right) dcdb + O(\log^2 x)\\
    &=\int_1^{\frac xB}   \int_1^{\frac xA \wedge \frac x{Bb}} \int_1^{\frac x{Ac}} \frac1{bca} \ dadcdb + O(\log^2 x)\\
    &=\iiint\limits_{\substack{{ac\leq x/A}\\{bc\leq x/B}}} \frac1{abc} dadbdc + O(\log^2 x).
    \end{align*}

    Thus, by Theorem 2.2, we have
    $$
    S_2(x)=\sum_{n_1, n_2\leq x} \frac1{[n_1,n_2]}=\frac2{\pi^2}\log^3 x + O(\log^2 x).
    $$
    Applying Dirichlet hyperbola method~\cite[(12.1.4)]{Ti} to the inner sum over $a_r$, $r\leq 3$, it is also possible to obtain a polynomial $P_3(x)$ of degree $3$ such that
    \begin{equation}S_2(x)=P_3(\log x)+O\left(\frac{\log^2 x}{\sqrt x}\right).\end{equation}
    A proof of (12) was given in~\cite[Theorem 2.1]{HLT}. The main method is Dirichlet hyperbola method, but the approach is different from ours.

    For $k=3$, we write $d=(n_1,n_2,n_3)$, $n_1=afgd$, $n_2=befd$, $n_3=cegd$, and $[n_1,n_2,n_3]=abcdefg$. This decomposition was used in~\cite{FF1}. Thus, $S_3(x)$ becomes
    $$
    \sum_{\substack{{a,b,c,d,e,f,g\leq x}\\{G\textrm{-wise coprime}}\\{afgd, befd, cegd\leq x}}} \frac1{abcdefg}
    $$
    where $G=(V,E)$ with $V=\{a,b,c,d,e,f,g\}$ and
    $$E=\{(a,b),(a,c),(a,e),(b,c),(b,g),(c,f),(e,g),(e,f),(f,g)\}.$$
    Therefore, by Theorem 2.2, we have
    $$
    S_3(x)=\rho(G)\textrm{vol}(D)\log^7 x+ O(\log^6 x).
    $$
    Here, $D$ is a $7$-dimensional polytope
    $$\{(a,b,c,d,e,f,g)\in [0,1]^7 | a+f+g+d\leq 1, b+e+f+d\leq 1, c+e+g+d\leq 1\}.$$
    The volume of this polytope is $11/3360$ by SageMath computation.
 \begin{figure}[!h]\includegraphics[width=6in]{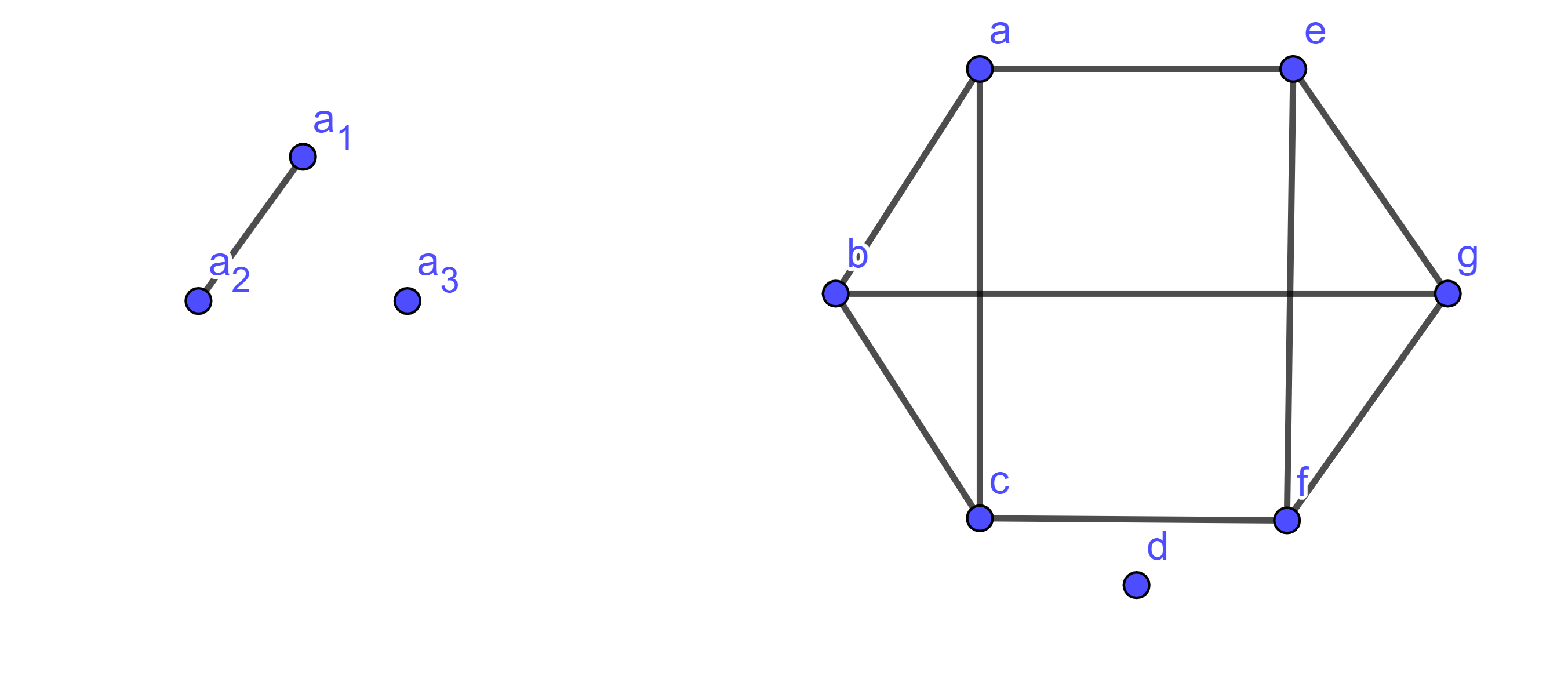}\caption{Coprimality graphs for $k=2,3$ (GeoGebra)}\end{figure}

    We compute $\rho(G)$ by Hu's expression. We have $i_0(G)=1$ due to the empty set, $i_1(G)=7$ due to the vertices, $i_2(G)=12$ due to the $6$ pairs of vertices without $d$ which do not form edges in $G$ and $6$ pairs of vertices with $d$, $i_3(G)=6$ due to adjoining $d$ to the $6$ pairs of vertices without $d$, and $i_m(G)=0$ if $m\geq 4$. We obtain an Euler product and an approximation by Python.
    $$
    \rho(G)=\prod_{p\in\mathcal{P}} \left(1-\frac9{p^2}+\frac{16}{p^3}-\frac9{p^4}+\frac1{p^6}\right)\approx 0.04932167.
    $$
    Therefore,
    $$
    \lim_{x\rightarrow\infty}\frac{S_3(x)}{\log^7 x}=\rho(G)\textrm{vol}(D)=\frac{11}{3360} \prod_{p\in\mathcal{P}} \left(1-\frac9{p^2}+\frac{16}{p^3}-\frac9{p^4}+\frac1{p^6}\right)\approx 0.00016147.
    $$
    \subsection{Coprimality graphs and the leading terms of $S_k(x)$, $U_k(x)$ and $V_k(x)$ for all $k\geq 2$}
    The coprimality graph for $k=2$ is $G_2=(\{1,2,3\}, \{(1,2)\})$ as in Figure 1. We construct $G_k$ inductively. Assume that $G_k=(V_k,E_k)=(\{1,\ldots, 2^k-1\}, \{\varepsilon_1,\ldots, \varepsilon_{e_k}\})$ is the coprimality graph for some $k\geq 2$. Regard the labels of vertices with $k$-bit binary strings that contains at least one $1$,
    $$V_k=\{1=0\cdots 01, 2=0\cdots 10, 3=0\cdots 11, \ldots, v_k=2^k-1=1\cdots 11\}.$$
 \begin{figure}[!h]\includegraphics[width=6in]{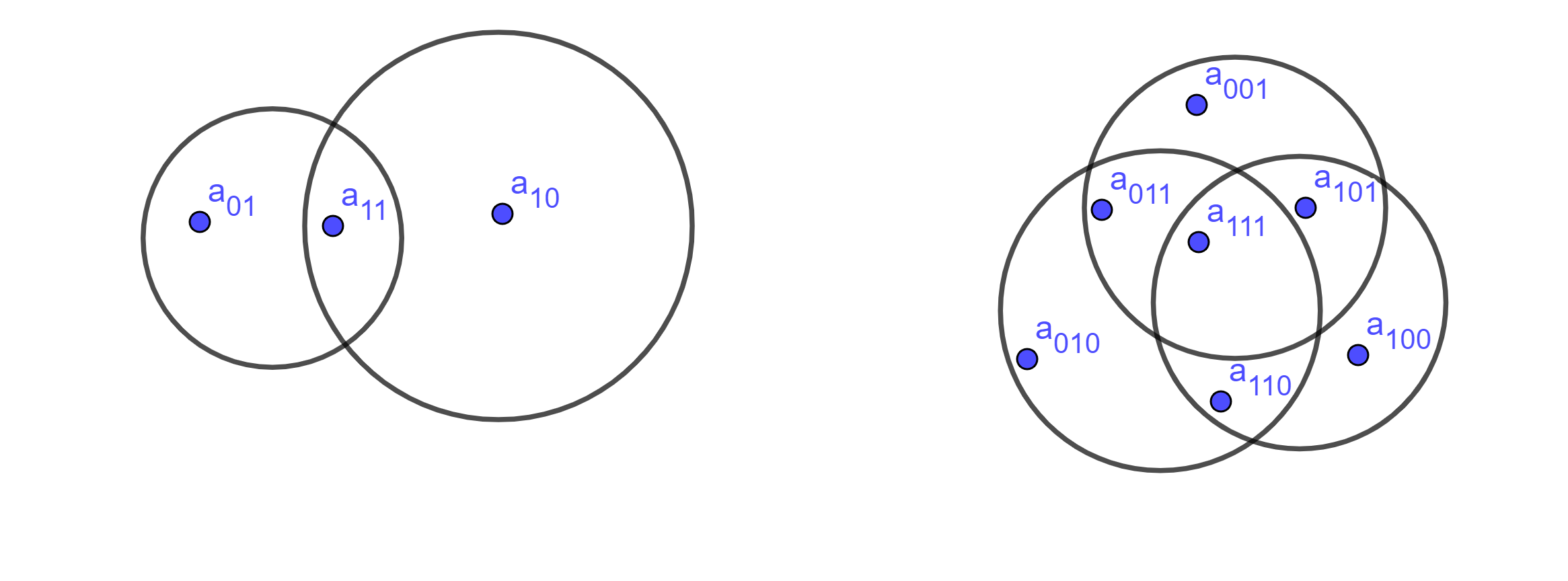}\caption{Vertex labeling for $k=2,3$ (GeoGebra)}\end{figure}
    The hyperbolic constraints $\{A_{k,i}\}_{i\leq k}$ are given by the rule that $A_{k,i}$ is the set of elements in $V_k$ such that $i$-th binary digit (from the right) is $1$. Then for each $i\leq k$, let
    $$
    n_i=\prod_{j\in A_{k,i}} a_j \leq x.
    $$
    Our inductive hypothesis is that under the coprimality conditions $E_k=\{\varepsilon_1,\ldots, \varepsilon_{e_k}\}$ imposed on $(a_1,\ldots, a_{v_k})\in (\N\cap[1,x])^{v_k}$, we have
    $$
    [n_1,\ldots, n_k]=\prod_{j\in V_k} a_j, \ \ (n_1,\ldots, n_k)=a_{v_k},
    $$
    hence
    $$
    S_k(x)=\sum_{n_1,\ldots, n_k\leq x} \frac1{[n_1,\ldots,n_k]}=\sum_{\substack{{a_1,\ldots, a_{v_k}\leq x}\\{G_k\textrm{-wise coprime}}\\{\forall_{i\leq k}, \prod_{j\in A_{k,i}}a_j\leq x}}}\frac1{a_1\cdots a_{v_k}},
    $$
    and
    $$
    U_k(x)=\sum_{\substack{{n_1,\ldots, n_k\leq x}\\{(n_1,\ldots, n_k)=1}}} \frac1{[n_1,\ldots, n_k]}=\sum_{\substack{{a_1,\ldots, a_{v_k}\leq x, \ a_{v_k}=1}\\{G_k\textrm{-wise coprime}}\\{\forall_{i\leq k}, \prod_{j\in A_{k,i}}a_j\leq x}}}\frac1{a_1\cdots a_{v_k}}.
    $$
    Let $V_{k+1}=\{1,\ldots, v_{k+1}=2^{k+1}-1\}$ be the set of ($k+1$)-bit binary strings that contains at least one $1$. Then we may also obtain $V_{k+1}$ by concatenating $1$ on the left for each element in $V_k$ and including $2^k=10\cdots 0$. Let $A_{k+1,i}=A_{k,i}\cup \{2^k+j \ | j\in A_{k,i}\}$ for each $i\leq k$ and let $A_{k+1,k+1}$ be the set of elements in $V_{k+1}$ starting with $1$. This process of concatenating and including one element can be viewed as introducing
    $$
    n_{k+1}=\prod_{j\in A_{k+1,k+1}} a_j \leq x.
    $$
    For each $i\leq k$, we redefine $n_i$'s as
    $$
    n_i=\prod_{j\in A_{k+1,i}} a_j \leq x.
    $$
    For each edge $(j,\ell)\in E_k$, let $e_{j,\ell}=\{2^k+j, j\}\times \{2^k+\ell, \ell\}$. We write
    $$
    E'_k=\bigcup_{(j,\ell)\in E_k} e_{j,\ell}.
    $$
    The edges in $E'_k$ extend the existing edges. By the inductive hypothesis,
    $$
    [n_1,\ldots, n_k]=\prod_{j\leq 2^{k+1}-1, j\neq 2^k} a_j.
    $$

    For each $i\leq k$, let $P(n_i,n_{k+1})$ be a product of $a_j$ where $j$'s binary digits start with $1$ and $i$-th digit (from the right) is $1$. That is,
    $$
    P(n_i,n_{k+1})=\prod_{j\in A_{k+1,k+1}\cap A_{k+1,i}}a_j.
    $$
    We want to have $P(n_i,n_{k+1})=(n_i,n_{k+1})$. To this end, we need to impose new coprimality conditions from
    $$
    \left(\frac{n_i}{P(n_i,n_{k+1})},\frac{n_{k+1}}{P(n_i,n_{k+1})}\right)=1, \ \forall i\leq k.
    $$
    Thus, the newly imposed coprimality conditions for each $i\leq k$ are between the following.\\

    $\mathbf{1.}$ $a_j$'s with $j$'s binary digits starting with $0$ and the $i$-th digit is $1$.

    $\mathbf{2.}$ $a_j$'s with $j$'s binary digits starting with $1$ and the $i$-th digit is $0$. \\

    Let $E'_{k+1}$ be the set of edges from these new conditions $(j_1,j_2)\in V_{k+1}^2$ where $j_1$ from {\bf 1} and $j_2$ from {\bf 2} so that
    $$
    E'_{k+1}=\bigcup_{i\leq k} \left(A_{k+1,k+1}^c \cap A_{k+1,i}\right) \times \left(A_{k+1,k+1} \cap A_{k+1,i}^c \right).
    $$
    Then we include these into the set of edges,
    $$
    E_{k+1}=E'_k\cup E'_{k+1}.
    $$
    By De Morgan's law,
    $$
    ([n_1,\ldots, n_k],n_{k+1})=[(n_1,n_{k+1}),\ldots, (n_k,n_{k+1})].
    $$
    By the edges in $E'_{k+1}$, we have for each $i\leq k$,
    $$
    (n_i,n_{k+1})=P(n_i,n_{k+1}).
    $$
    By the edges in $E'_k$ and the inductive hypothesis, we have
    $$
    ([n_1,\ldots, n_k],n_{k+1})=[P(n_1,n_{k+1}),\ldots, P(n_k,n_{k+1})]=\prod_{2^k+1\leq j\leq 2^{k+1}-1} a_j.
    $$
    Hence, we obtain
    $$
    [n_1,\ldots, n_{k+1}]=\left[[n_1,\ldots, n_k],n_{k+1}\right]=\frac{[n_1,\ldots,n_k] n_{k+1}}{\left([n_1,\ldots,n_k], n_{k+1}\right)}=\prod_{j\in V_{k+1}} a_j.
    $$
    Moreover, we obtain by the edges in $E'_{k+1}$,
    $$
    (n_1,\ldots, n_{k+1})=\left((n_1,\ldots, n_k), n_{k+1}\right)=a_{2^{k+1}-1}.
    $$
    We have constructed the coprimality graph $G_{k+1}$ so that
    $$
    S_{k+1}(x)=\sum_{n_1,\ldots, n_{k+1}\leq x} \frac1{[n_1,\ldots,n_{k+1}]}=\sum_{\substack{{a_1,\ldots, a_{v_{k+1}}\leq x}\\{G_{k+1}\textrm{-wise coprime}}\\{\forall_{i\leq k+1}, \prod_{j\in A_{k+1,i}}a_j\leq x}}}\frac1{a_1\cdots a_{v_{k+1}}}
    $$
    and
    $$
    U_{k+1}(x)=\sum_{\substack{{n_1,\ldots, n_{k+1}\leq x}\\{(n_1,\ldots, n_{k+1})=1}}} \frac1{[n_1,\ldots, n_{k+1}]}=\sum_{\substack{{a_1,\ldots, a_{v_{k+1}}\leq x, \ a_{v_{k+1}}=1}\\{G_{k+1}\textrm{-wise coprime}}\\{\forall_{i\leq k+1}, \prod_{j\in A_{k+1,i}}a_j\leq x}}}\frac1{a_1\cdots a_{v_{k+1}}}.
    $$
    This completes the construction of $n_i$, $i\leq k$ through $a_j$'s so that the coprimality conditions are satisfied. Let $j_i$ be the $i$-th binary digit from the right. We may also obtain $a_j$'s directly in terms of $n_i$, $i\leq k$ by setting
    $$
    a_j=\prod_{p\in\mathcal{P}} p^{\nu_p(a_j)}
    $$
    where $\nu_p(a_j)$ is the Euclidean length of the interval
    $$
    \bigcap_{\substack{{i\leq k}\\{j_i=1}}}[0,\nu_p(n_i)]\bigcap_{\substack{{i\leq k}\\{j_i=0}}} [0, \nu_p(n_i)]^c.$$
    Thus,
    $$
    \nu_p(a_j)= \min_{\substack{{i\leq k}\\{j_i=1}}} \nu_p(n_i)-\max_{\substack{{i\leq k}\\{j_i=0}}} \nu_p(n_i).
    $$
    Then the coprimality conditions listed in $E_k$ are satisfied.

    Hence, the leading term of (6) follows from Theorem 2.2 with $c_k=\rho(G_k)\mathrm{vol}(D_k)$ where
    $$
    D_k=\left\{(u_1,\ldots u_{v_k})\in [0,\infty)^{v_k} \bigg\vert \forall_{i\leq k}, \sum_{j\in A_{k,i}} u_j \leq 1\right\}.
    $$
    That is,
    $$
    S_k(x)=\rho(G_k)\mathrm{vol}(D_k) \log^{2^k-1}x + O(\log^{2^k-2}x).
    $$
    To find the leading term of (7), let $V_k-\{v_k\}=V_k^*$ and $G_k^*=(V_k^*,E_k)$. Our construction guarantees that $E_k$ does not contain any edges connecting to the vertex $v_k$. The set of edges of $G_k^*$ and that of $G_k$ are identical. Then we have by~\cite{RH},
    $$
    \rho(G_k)=\rho(G_k^*).
    $$
    For $U_k(x)$, we use
    $$
    U_k(x)=\sum_{\substack{{n_1,\ldots, n_k\leq x}\\{(n_1,\ldots, n_k)=1}}} \frac1{[n_1,\ldots, n_k]}=\sum_{\substack{{a_1,\ldots, a_{v_k}\leq x, \ a_{v_k}=1}\\{G_k\textrm{-wise coprime}}\\{\forall_{i\leq k}, \prod_{j\in A_{k,i}}a_j\leq x}}}\frac1{a_1\cdots a_{v_k}}.
    $$
    The sum on the right side equals
    $$
    \sum_{\substack{{a_1,\ldots, a_{v_k-1}\leq x}\\{G_k^*\textrm{-wise coprime}}\\{\forall_{i\leq k}, \prod_{j\in A_{k,i}^*}a_j\leq x}}}\frac1{a_1\cdots a_{v_k-1}}
    $$
    where $A_{k,i}^*=A_{k,i}-\{v_k\}$ for each $i\leq k$. By Theorem 2.2, the sum equals
    $$
    U_k(x)=\rho(G_k^*)\textrm{vol}(D_k^*) \log^{2^k-2}x+O(\log^{2^k-3}x)
    $$
    where
    $$
    D_k^*=\left\{(u_1,\ldots u_{v_k-1})\in [0,\infty)^{v_k-1} \bigg\vert \forall_{i\leq k}, \sum_{j\in A_{k,i}^*} u_j \leq 1\right\}.
    $$
    Interchanging the volume integral for $D_k$ and put the integral over $u_{v_k}$ the leftmost, we have
    $$
    D_k=\int_0^1 (1-u_{v_k})^{v_k-1}\textrm{vol}(D_k^*)\ du_{v_k} = \frac1{v_k} \textrm{vol}(D_k^*).
    $$
    Hence, the leading term of (7) follows.

    Clearly for any $k\geq 2$, we have $\rho(G_k)>\rho(G_{k+1})$ by observing that
    $$\rho(G_k)=\rho( (V_{k+1}, E_k)) > \rho( (V_{k+1}, E_{k+1}))=\rho(G_{k+1}).$$
    Consider a standard embedding of $D_k$ into $[0,\infty)^{v_{k+1}}$. We see that
    $$\left(\forall_{i\leq k+1}, \sum_{j\in A_{k+1,i}}u_j\leq 1 \right)\textrm{ implies } \left(\forall_{i\leq k}, \sum_{j\in A_{k,i}} u_j\leq 1 \textrm{ and }\sum_{j\in A_{k+1,k+1}}u_j\leq 1\right).$$
    Thus, for any $k\geq 2$, we have $\textrm{vol}(D_{k+1})\leq \frac1{(2^k)!}\textrm{vol}(D_k)$, yielding $c_k>c_{k+1}$. This implies
    $$
    \mathrm{vol}(D_k)\leq \frac1{(2^{k-1})! (2^{k-2})! \cdots (2^2)!\cdot 3}.
    $$
    Therefore, $c_k\rightarrow 0$ quite rapidly as $k\rightarrow\infty$.

    To obtain the leading term of $V_k(x)$, we begin with isolating vertices labeled as powers of $2$. We write $j\neq 2^t$ to indicate that $j$ is not a power of $2$. Then we have

    \begin{align*}
    V_k(x)&=\sum_{n_1,\ldots, n_k\leq x} \frac{n_1\cdots n_k}{[n_1,\ldots, n_k]}\\
    &=\sum_{\substack{{a_1,\ldots, a_{v_k}\leq x} \\{G_k\textrm{-wise coprime}}\\{\forall_{i\leq k},\prod_{j\in A_{k,i}}a_j\leq x}}}\frac{\prod\limits_{i\leq k} \prod\limits_{j\in A_{k,i}}a_j}{a_1\ldots a_{v_k}}=\sum_{\substack{{a_1,\ldots, a_{v_k}\leq x} \\{G_k\textrm{-wise coprime}}\\{\forall_{i\leq k},\prod\limits_{j\in A_{k,i}}a_j\leq x}}}\frac{\prod\limits_{i\leq k} \prod\limits_{j\in A_{k,i}-\{2^{i-1}\}}a_j}{\prod\limits_{\substack{{j\in V_k}\\{j\neq 2^t}}} a_j}\\
    &=\sum_{\substack{{\forall_{j\leq v_k, j\neq 2^t}, a_j\leq x} \\{G_k\textrm{-wise coprime}}\\{\forall_{i\leq k},\prod\limits_{j\in A_{k,i}-\{2^{i-1}\}}a_j\leq x}}}\frac{\prod\limits_{i\leq k} \prod\limits_{j\in A_{k,i}-\{2^{i-1}\}}a_j}{\prod\limits_{\substack{{j\in V_k}\\{j\neq 2^t}}} a_j}\prod_{i\leq k}\sum_{a_{2^{i-1}}\leq x/\prod\limits_{j\in A_{k,i}-\{2^{i-1}\}} a_j} 1\\
    &=\sum_{m_1,\ldots, m_{e_k}=1}^{\infty}\prod\limits_{j\leq e_k}\mu(m_j)\sum_{\substack{{\forall_{j\leq v_k, j\neq 2^t}, M_jb_j\leq x} \\{\forall_{i\leq k},\prod\limits_{j\in A_{k,i}-\{2^{i-1}\}}(M_jb_j)\leq x}}}\frac{\prod\limits_{i\leq k} \prod\limits_{j\in A_{k,i}-\{2^{i-1}\}}(M_jb_j)}{\prod\limits_{\substack{{j\in V_k}\\{j\neq 2^t}}} (M_jb_j)}\prod_{i\leq k}\sum_{b_{2^{i-1}}\leq \frac x{M_{2^{i-1}}\prod\limits_{j\in A_{k,i}-\{2^{i-1}\}} (M_jb_j)}} 1
    \end{align*}
    By writing
    $$
    \sum_{b_{2^{i-1}}\leq \frac x{M_{2^{i-1}}\prod\limits_{j\in A_{k,i}-\{2^{i-1}\}} (M_jb_j)}} 1=\frac x{M_{2^{i-1}}\prod\limits_{j\in A_{k,i}-\{2^{i-1}\}} (M_jb_j)}+O\left(\left[\frac x{M_{2^{i-1}}\prod\limits_{j\in A_{k,i}-\{2^{i-1}\}} (M_jb_j)}\right]^{1-\alpha}\right)
    $$
    for $\alpha<1/2$, the $O$-terms (without loss of generality) contribute the following before summing over $m_i$'s.
    \begin{align*}
    \sum_{\forall_{i\leq k},\prod\limits_{j\in A_{k,i}-\{2^{i-1}\}}(M_jb_j)\leq x}&\frac {x^{k-\alpha}}{\left[M_1 \prod\limits_{j\in A_{k,1}-\{1\}} (M_jb_j)\right]^{1-\alpha}\prod\limits_{\substack{{j\in V_k-A_{k,1}}\\{j\neq 2^t}}}(M_jb_j) \prod\limits_{\substack{{j=2^t}\\{j>1}}}M_j}\\
    &=O\left(\frac{x^{k-\alpha}\log^{2^{k-1}-k} x\sum\limits_{n\leq x}n^{\alpha-1}\tau_{2^{k-1}-1}(n)}{\left[M_1\prod\limits_{j\in A_{k,1}-\{1\}}M_j\right]^{1-\alpha}\prod\limits_{\substack{{j\in V_k-A_{k,1}}\\{j\neq 2^t}}}M_j \prod\limits_{\substack{{j=2^t}\\{j>1}}}M_j}\right)\\
    &=O\left(\frac{x^k\log^{2^k-k-2} x}{\left[\prod\limits_{j\in A_{k,1}}M_j\right]^{1-\alpha}\prod\limits_{j\in V_k-A_{k,1}}M_j}\right).
    \end{align*}
    We have the convergence of the sum over $m_i$'s due to $\alpha<1/2$. Then we have $V_k(x)$ up to an error of $O(x^k \log^{2^k-k-2} x)$:
    \begin{align*}
    V_k(x)&=x^k\sum_{m_1,\ldots, m_{e_k}=1}^{\infty} \frac{\prod\limits_{j\leq e_k}\mu(m_j)}{\prod\limits_{j\leq v_k}M_j} \sum_{\forall_{i\leq k},\prod\limits_{j\in A_{k,i}-\{2^{i-1}\}}(M_jb_j)\leq x} \frac1{\prod\limits_{\substack{{j\in V_k}\\{j\neq 2^t}}}b_j}+O(x^k \log^{2^k-k-2} x)\\
    &=x^k  \rho(G_k) \mathrm{vol}(D^{**}_k) \log^{2^k-k-1}x +O(x^k\log^{2^k-k-2} x)
    \end{align*}
    where $D^{**}_k$ is a convex polytope defined by the hyperbolic constraints $\{A_i-\{2^{i-1}\}\}_{i\leq k}$,
    $$
    D^{**}_k=\left\{(t_i)_{\substack{{i\leq {v_k}}\\{i\neq 2^t}}}\in [0,\infty)^{2^k-k-1} \ \bigg\vert  \forall_{i\leq k}, \ \sum_{j\in A_{k,i}-\{2^{i-1}\}}t_j \leq 1 \right\}.
    $$

    \subsection{Multivariable Perron's formula}
    A multivariable Perron's formula relates the sum of multivariable arithmetic function
    $$
    \sum_{n_1,\ldots, n_k\leq x} f(n_1,\ldots, n_k)
    $$
    to an integral over some vertical lines of complex planes. A general result yielding an asymptotic formula with power-saving error term was proved by de la Bret\`{e}che~\cite[Theorem 1]{dlB}. An effective version of Perron's formula was proved by Balazard, Naimi, P\'{e}termann~\cite[Proposition 5]{BNP} in which they provided an asymptotic formula for the sum
    $$
    \sum_{n_1,\ldots, n_k\leq x} \frac{\mu(n_1)\cdots \mu(n_k)}{[n_1,\ldots,n_k]}.
    $$
    T\'{o}th and Zhai~\cite{TZ} proved an asymptotic formula for the average number of subgroups of $\Z/m\Z\times\Z/n\Z$, $m,n\leq x$ by using the effective Perron's formula.
    We state the effective version of Perron's formula~\cite[Proposition 5]{BNP} here.
    \begin{prop}
    Let $f:\N^k \rightarrow \C$ be an arithmetic function of $k$ variables and $(\sigma_{a_1}, \ldots, \sigma_{a_k})$ an $k$-tuple of the abscissas of absolute convergence of the associated Dirichlet series
    $$F(s_1,\ldots, s_k)=\sum_{n_1,\ldots, n_k=1}^{\infty}\frac{f(n_1,\ldots, n_k)}{{n_1}^{s_1}\cdots {n_k}^{s_k}}.$$ We have for $x_1\geq 1$, $\ldots \ $, $x_k\geq 1$, $\kappa_1>\max(0,\sigma_{a_1})$, $\ldots \ $, $\kappa_k>\max(0,\sigma_{a_k})$, $T_1\geq 1$, $\ldots \ $, $T_k\geq 1$,
    $$
    \left|\sideset{}{'}\sum_{n_1\leq x_1, \ldots, n_k\leq x_k}  f(n_1,\ldots, n_k)-\frac1{(2\pi i)^k}\int\displaylimits_{\kappa_k-iT_k}^{\kappa_k+iT_k}\cdots \int\displaylimits_{\kappa_1-iT_1}^{\kappa_1+iT_1}F(s_1,\ldots, s_k)\frac{{x_1}^{s_1}\cdots {x_k}^{s_k}}{s_1\cdots s_k} ds_1\cdots ds_k\right|$$
    \begin{align*}\leq x_1^{\kappa_1}\cdots x_k^{\kappa_k} &\sum_{n_1,\ldots, n_k=1}^{\infty} \frac{|f(n_1,\ldots, n_k)|}{{n_1}^{\kappa_1}\cdots {n_k}^{\kappa_k}}
    Q_k\left(\frac1{(\pi T_1 |\log(x_1/n_1)|)^*},\ldots,\frac1{(\pi T_k |\log(x_k/n_k)|)^*}\right)
    \end{align*}
    where the primed sum takes the value of $f(n_1,\ldots, n_k)/2^m$ if the number of $i\leq k$ with $n_i=x_i$ is m, the multivariable polynomial $Q_k$ is defined by
    $$
    Q_k(X_1,\ldots, X_k)=(1+X_1)\cdots (1+X_k)-1,
    $$
    and $y^*=\max(1,y)$.
    \end{prop}
    We are interested in the case where $f$ takes the following form
    \begin{equation}
    f(n_1,\ldots, n_k)=\sum_{\substack{{\forall_r, a_r\in \N}\\{G\textrm{-wise coprime}}\\{\forall_{i\leq k}, \ \prod_{j\in A_i} a_j = n_i}}}\frac1{a_1\cdots a_v}
    \end{equation}
    where $G=(V,E)=(\{1,\ldots, v\}, \{\epsilon_1,\ldots, \epsilon_e\})$ is a graph and $A_i, i\leq k$ are nonempty with $\cup_{i\leq k} A_i = \{1, \ldots, v\}$. The associated Dirichlet series of the arithmetic function $f$ above satisfies the following.
    \begin{lemma}
    Let $f$ be the arithmetic function defined in (13), ${\bf s}=(s_1,\ldots, s_k)\in \C^k$, and ${\bf e}_i, i\leq k$ be the standard basis of $\C^k$. Then
    \begin{equation}
    F(s_1,\ldots, s_k)=\sum_{n_1,\ldots, n_k=1}^{\infty}\frac{f(n_1,\ldots, n_k)}{{n_1}^{s_1}\cdots {n_k}^{s_k}}=H(s_1,\ldots, s_k)\prod_{r\leq v} \zeta(1+\ell_r({\bf s}))
    \end{equation}
    converges absolutely if $\Re(s_i)>0$ for each $i\leq k$, the Dirichlet series of $H$ converges absolutely if $\Re(s_i)>-\frac1{2k}$ for each $i\leq k$, and the linear forms $\ell_r({\bf s})$ is defined by
    $$
    \ell_r({\bf s})=\sum_{i\leq k} \ell_r({\bf e}_i) s_i
    $$
    where
    $$
    \ell_r({\bf e}_i)=\begin{cases} 1 &\mbox{ if }r\in A_i\\ 0 &\mbox{ otherwise.}\end{cases}
    $$
    \end{lemma}
    \begin{proof}
    The absolute convergence of $F$ when $\Re(s_i)>0$ for each $i\leq k$ is clear.
    We have by the method of Section 2,
    \begin{align*}
    \sum_{n_1,\ldots, n_k=1}^{\infty}&\frac{f(n_1,\ldots, n_k)}{{n_1}^{s_1}\cdots {n_k}^{s_k}}=\sum_{\substack{{\forall_r, a_r\in \N}\\{G\textrm{-wise coprime}}}}\frac1{a_1^{1+\ell_1({\bf s})}\cdots a_v^{1+\ell_v({\bf s})}}\\
    &=\sum_{m_1,\ldots, m_e=1}^{\infty} \frac{\mu(m_1)\cdots \mu(m_e)}{M_1^{1+\ell_1({\bf s})}\cdots M_v^{1+\ell_v({\bf s})}}\sum_{b_1,\ldots, b_v=1}^{\infty} \frac1{b_1^{1+\ell_1({\bf s})}\cdots b_v^{1+\ell_v({\bf s})} }.
    \end{align*}
    The inner sum over $b_i$ is $\prod_{r\leq v} \zeta(1+\ell_r({\bf s}))$. By the analysis of the multiplicative function $f_G(m)$ in the beginning of Section 3, the outer sum over $m_i$ defines a Dirichlet series which converges absolutely if $\Re(1+\ell_r({\bf s}))>1/2$ for each $r\leq v$. This is satisfied if $\Re(s_i)>-\frac1{2k}$ for each $i\leq k$. Denote this outer sum by $H(s_1,\ldots, s_k)$, then we have the result.
    \end{proof}
    The following are the well-known upper bound and mean value results for the Riemann zeta function. Similar results were also used in~\cite{TZ}.
    \begin{lemma}
    Fix $\delta>0$, $M, N\geq 0$, and let $s=\sigma+it$. For any $\epsilon>0$, there is an absolute constant $A>0$ depending on $\delta$, $\epsilon$, $M$, and $N$ such that for all $j\leq N$, $1/2+\delta\leq \sigma \leq 1+\delta$, $t\geq 1$, and $T\geq 1$,
    \begin{equation}
    |\zeta^{(j)} (s)|\leq A |t|^{\frac13 \max(1-\sigma,0)+\epsilon}
    \end{equation}
    and
    \begin{equation}
    \int_{-T}^{T} |\zeta^{(j)}(\sigma+it)|^M dt \leq A \ T^{\frac{\max(M-4,0)}3\max(1-\sigma,0)+1+\epsilon}.
    \end{equation}
    \end{lemma}
    \begin{proof}
    If $j=0$, (15) follows from~\cite[(5.12)]{Ti} and Phragm\'{e}n-Lindel\"{o}f principle. For $j>0$, the result follows by Cauchy's integral formula. If $j=0$, use
    $$
    |\zeta(\sigma+it)|^M=|\zeta(\sigma+it)|^{\max(M-4,0)} |\zeta(\sigma+it)|^{\min(M,4)}.
    $$
    Then (16) follows from~\cite[Theorem 7.5]{Ti} and the analogous result for the mean values in~\cite[page 149]{Ti} similar to Phragm\'{e}n-Lindel\"{o}f principle. Again, the result for $j>0$ follows by Cauchy's integral formula.
    \end{proof}
    For the values of $\zeta^{(j)}(s)$ when $s$ is close to the line $\Re(s)=1$, we have the following result.
    \begin{lemma}
    Let $A$ be any positive constant. Fix $N\geq 0$, and let $s=\sigma+it$ with $t\geq t_0$. Then there are absolute constants $B>0$ and $M>0$ depending on $A$ and $N$ such that for all $j\leq N$,
    $$
    |\zeta^{(j)}(s)|\leq B \log^M t
    $$
    uniformly in the region
    $$
    1-\frac A{\log t}\leq \sigma \leq 2, \ t\geq t_0.
    $$
    \end{lemma}
    \begin{proof}
    For $j=0$, this is~\cite[Theorem 3.5]{Ti}. For $j=1$, the proof is in~\cite[(3.5.5)]{Ti}. The proofs for the higher derivatives are analogous.
    \end{proof}
    We establish a version of Perron's formula for the Dirichlet series (14). The following is a special case of~\cite[Theorem 1]{dlB}. The purpose of the following proposition is to find an explicit power-saving error term with the exponent of $x$ is depending on $k$ and $v$. Our assumptions are more restrictive than that of~\cite[Theorem 1]{dlB}. We use the notations ${\bf n}=(n_1,\ldots, n_k)\in \N^k$, $\boldsymbol{\alpha}=(\alpha_1,\ldots, \alpha_k)\in [0,\infty)^k$, ${\bf 0}=(0,\ldots, 0)\in\C^k$, ${\bf 1}=(1,\ldots, 1)\in \C^k$, ${\bf n}^{\boldsymbol{\alpha}}={n_1}^{\alpha_1}\cdots {n_k}^{\alpha_k}$, and $J(\boldsymbol{\alpha})=\{j\leq k \ | \ \alpha_j=0\}$.
    \begin{prop}
    Let $f:\N^k\rightarrow \R$ be an arithmetic function satisfying the following conditions.

    i) The associated Dirichlet series
    $$
    F(s_1,\ldots, s_k)=\sum_{n_1,\ldots, n_k=1}^{\infty}\frac{f(n_1,\ldots, n_k)}{{n_1}^{s_1}\cdots {n_k}^{s_k}}=H(s_1,\ldots, s_k)\prod_{r\leq v} \zeta(1+\ell_r({\bf s})),
    $$
    converges absolutely if $\Re(s_i)>0$ for each $i\leq k$, the Dirichlet series of $H$ converges absolutely if $\Re(s_i)>-\frac1{2k}$ for each $i\leq k$, and $\ell_r({\bf s})$, $r\leq v$ are nonzero linear forms.

    ii) There is a constant $M=M_f\in\N$ such that for all $i\leq k$ and $0<\kappa_j<1$,
    \begin{equation}
    \sum_{\substack{{\forall_{j\leq k}, n_j\in\N}\\{j\neq i}}} \frac{|f(n_1,\ldots, n_{i-1}, n_i, n_{i+1},\ldots, n_k)|}{{n_1}^{\kappa_1}\cdots {n_{i-1}}^{\kappa_{i-1}}{n_{i+1}}^{\kappa_{i+1}}\cdots {n_k}^{\kappa_k}} \leq \frac{M\tau_v(n_i)}{(\kappa_1\cdots \kappa_{i-1})^M  \ n_i \ (\kappa_{i+1}\cdots \kappa_k)^M}.\end{equation}
    where $\tau_v(m)=\sum_{a_1\cdots a_v=m} 1$ is the divisor function.

    Then we have
    $$
    \sum_{n_1,\ldots, n_k\leq x} {\bf n}^{\boldsymbol{\alpha}}f(n_1,\ldots, n_k)=x^{\boldsymbol{\alpha}\cdot {\bf 1}} P_{v+w-\ell}(\log x)+O(x^{\boldsymbol{\alpha}\cdot {\bf 1}-\frac{v_*}C+\epsilon})
    $$
    where $w=|J(\boldsymbol{\alpha})|$, $\ell=\mathrm{rank}\left(\{\ell_r \ | \ r\leq v\} \cup \{{\bf e}_j^* \ | \ j\in J(\boldsymbol{\alpha}), \ j\leq k\}\right)$, $v_*=\frac 3{6k+\max(v-4,0)}$,

    Here, $C\geq 1$ will be defined in the proof.
    \end{prop}
    \begin{proof}
    We first treat the case $\boldsymbol{\alpha} = {\bf 0 }\in [0,\infty)^k$. We apply Proposition 3.1 with $2\leq x_1=\cdots=x_k=x\in \N$, $2\leq T_1=\cdots=T_k=T\leq x$, and $0<\kappa_j< \frac1{\log x}$ for each $j\leq k$. We treat the error term in Proposition 3.1. Without loss of generality, consider a monomial term $X_1\cdots X_r$ of $Q_k$. By $0<1/y^*\leq 1$ and (17), the contribution of the terms such that some $n_i$, $i\leq r$ do  not belong to $(x/2, 2x)$ is
    $$
    \ll x^{\kappa_1+\cdots + \kappa_k} \sum_{n_1,\ldots, n_k=1}^{\infty} \frac{|f(n_1,\ldots, n_k)|}{{n_1}^{\kappa_1} \cdots {n_k}^{\kappa_k}T}\ll \frac{\Lm}T.
    $$
    For the remaining terms, without loss of generality, consider the sum over $n_1\in (x/2, 2x)-\{x\}$, $n_2,\ldots, n_k \in \N$. By $|\log(x/n)|\asymp |x-n|/x$ for $n\asymp x$ and (17), the sum of such terms satisfy
    \begin{align*}
    &\ll x^{\kappa_1+\cdots + \kappa_k} \sum_{\substack{{x/2<n_1<2x}\\{n_1\neq x}}} \sum_{n_2,\ldots, n_k=1}^{\infty} \frac{|f(n_1,\ldots, n_k)|}{{n_1}^{\kappa_1} \cdots {n_k}^{\kappa_k}}\frac x{T|x-n_1|}\\
    &\ll \sum_{\substack{{x/2<n_1<2x}\\{n_1\neq x}}} \frac{\tau_v(n_1)x\Lm}{Tn_1|x-n_1|} \ll \sum_{\substack{{x/2<n_1<2x}\\{n_1\neq x}}} \frac{\tau_v(n_1)\Lm}{T|x-n_1|}\ll \frac{x^{\epsilon}}T.
    \end{align*}
    With some more work (see~\cite[Lemma 1, 2, and 3]{BNP}), this error from the remaining terms can be improved to $O(\Lm/T)$ provided that $T\leq x^{v_*-\epsilon}$.

    The term with $n_1=x$, is by $0<1/y^*\leq 1$ and (17),
    \begin{align*}
    &\ll x^{\kappa_1+\cdots + \kappa_k} x^{-1+\epsilon}\ll x^{-1+\epsilon}.
    \end{align*}
    Therefore, the combined error term becomes $\ll T^{-1}\Lm$ when $2\leq T\leq x^{v_*-\epsilon}$.

    Also, by the assumption (17) on $f$,
    $$
    \sideset{}{'}\sum_{n_1,\ldots, n_k\leq x} f(n_1,\ldots, n_k) = \sum_{n_1,\ldots, n_k\leq x} f(n_1,\ldots, n_k)+O(x^{-1+\epsilon}).
    $$
    By Proposition 3.1, we have
    $$
    \sum_{n_1,\ldots, n_k\leq x} f(n_1,\ldots, n_k)=\frac1{(2\pi i)^k}\int\displaylimits_{\kappa_k-iT_k}^{\kappa_k+iT_k}\cdots \int\displaylimits_{\kappa_1-iT_1}^{\kappa_1+iT_1}F(s_1,\ldots, s_k)\frac{x^{s_1+\cdots+s_k}}{s_1\cdots s_k} ds_1\cdots ds_k+O(\frac{\Lm}T).
    $$
    As in~\cite[(2.11)]{dlB}, we choose $\kappa_j=\kappa^j/\log x$ where $\kappa$ is a positive transcendental over the smallest field $K$ containing $\Q$ and all coefficients of $\ell_r({\bf s})$. This is to ensure any distinct linear forms that appear in evaluating the above integral have distinct values.

    For the general $\boldsymbol{\alpha}$, we obtain by applying Proposition 3.1 to the function $$F(s_1-\alpha_1, \ldots, s_k-\alpha_k)$$ and change of variables,
    $$
    \sum_{n_1,\ldots, n_k\leq x} {\bf n}^{\boldsymbol{\alpha}}f({\bf n})=\frac{x^{\boldsymbol{\alpha}\cdot {\bf 1}}}{(2\pi i)^k}\int\displaylimits_{\kappa_k-iT_k}^{\kappa_k+iT_k}\cdots \int\displaylimits_{\kappa_1-iT_1}^{\kappa_1+iT_1}\frac{F(s_1,\ldots, s_k)x^{s_1+\cdots+s_k}}{(s_1+\alpha_1)\cdots (s_k+\alpha_k)} ds_1\cdots ds_k+O(\frac{x^{\boldsymbol{\alpha}\cdot {\bf 1}}\Lm}T)
    $$
    with the same choice of $\kappa_j$.

    The steps of evaluating this integral are written in~\cite{dlB} and~\cite{BNP}. Here, we include the outline of the method. The idea is to evaluate the integral starting from the innermost by moving the line of integration to $[-\delta_1 - iT_1, -\delta_1+ iT_1]$ and using the residue theorem. We choose $\delta_1$ and subsequent $\delta_j$'s so that the zeta functions are evaluated close to the critical line and the mean value result (Lemma 3.2) are applied. Note that $\zeta(1+s)$ has a simple pole at $s=0$.  \\

    $\mathbf{1.}$ Consider a $(v+w+k+1)\times k$ matrix $L=(a_{ij})_{\substack{{1\leq i\leq v+w+k+1}\\{1\leq j\leq k}}}$ obtained by the coefficients of $\ell_r$, $r\leq v$ on the first $v$ rows, and ${\bf e}_j, \ j\in J(\boldsymbol{\alpha})$ on the next $w$ rows. Then we put the $k\times k$ identity matrix below. The last row is ${\bf 1}\in \C^k$. Let $L_0$ be the top $v+w$ rows of $L$. \\

    $\mathbf{2.}$ Choose a row with nonzero first entry (if exists, and it is not the last $k+1$ rows). Apply the row interchange to move this row to the first row $R_1$. Apply the row interchange to move all remaining rows $R_i$, $i\leq v+w$ with nonzero first entry to the rows directly below $R_1$.  By taking the dot product with $(0,-s_2,\ldots, -s_k)$ and dividing by the nonzero first entry of these rows, we obtain the locations of poles of the integrand with respect to the variable $s_1$. The poles are linear forms of $s_2,\ldots, s_k$. \\

    $\mathbf{3.}$ We find the residues at these poles. The sum of the residues is the value of the innermost integral up to an error caused by two horizontal sides and the left vertical side of the rectangle with vertices $-\delta_1-iT_1$, $\kappa_1-iT_1$, $\kappa_1+iT_1$, and $-\delta_1+iT_1$.  \\

    $\mathbf{4.}$ We need to substitute the chosen pole in place of $s_1$ into the integrand except for the function corresponding to the chosen pole. We apply the row replacement of $R_i$ by $R_i-\frac{a_{i1}}{a_{11}}R_1$ to all the rows below $R_1$. This will remove the leading nonzero entries of these rows except for $R_1$. Then we proceed on integrating the residue with respect to $s_2$.  \\

    $\mathbf{5.}$ Suppose that we are at the stage of the integral with respect to $s_j$, $j\geq 2$ and the last row shows $s_1+\cdots+ s_k$ after the substitutions of $s_1, \ldots, s_{j-1}$. The $k$ rows above the last row show the substitutions of $s_1, \ldots, s_k$ in terms of $s_j, \ldots, s_k$.\\

        {\bf Case 1}: The first nonzero entry in the last row is at the position $j_0\geq j$ and it is positive. \\

        We interchange the order of integration from $ds_j\cdots ds_k$ to
        $$ds_{j_0}ds_jds_{j+1}\cdots ds_{j_0-1}ds_{j_0+1}\cdots ds_k$$
        with an obvious modification when $j_0=j, k$. The poles with respect to $s_{j_0}$ are linear forms of $s_{j},\ldots, s_{j_0-1},s_{j_0+1},\ldots, s_k$. We find the residues at these poles. Move the line of integration to $[-\delta_{j_0}-iT_{j_0}, -\delta_{j_0}+iT_{j_0}]$. The error caused by two horizontal sides and the left vertical sides of the rectangle with vertices $-\delta_{j_0}-iT_{j_0}, \kappa_{j_0}-iT_{j_0}, \kappa_{j_0}+iT_{j_0},$ and $-\delta_{j_0}+iT_{j_0}$.\\

        {\bf Case 2}: The first nonzero entry in the last row is at the position $j_0\geq j$ and it is negative. \\

        We interchange the order of integration as in Case 1. We find the residues at these poles. Move the line of integration to $[\delta_{j_0}-iT_{j_0}, \delta_{j_0}+iT_{j_0}]$. The error caused by two horizontal sides and the right vertical sides of the rectangle with vertices $\delta_{j_0}-iT_{j_0}, \kappa_{j_0}-iT_{j_0}, \kappa_{j_0}+iT_{j_0},$ and $\delta_{j_0}+iT_{j_0}$.\\

        In Case 1 and Case 2, we modify $T_{j_0}$ by $T_{j_0}+O(1)$ to avoid the poles in the path of integration if necessary. This may result in an error of $O(\Lm/T)$ by Lemma 3.3. We may not have all poles enclosed by the rectangle. However, it is possible to include all poles on one side of the vertical line $\Re s = \kappa_{j_0}$ at a cost of an error $O(\Lm/T)$ due to $(s_{j_0}+\alpha_{j_0})^{-1}$ in the integrand and by Lemma 3.3. The interchange of the order of integration causes the reordering of the columns of $L$ to be consistent with the order of variables. The steps 2, 3, 4, and 5 are repeated until the following Case 3 appears. If Case 3 has never appeared, then the process ends with contour shifting and finding residues at poles about the last variable of consideration. \\

        {\bf Case 3}: The $j_0$-th position of the last row is zero for all $j_0\geq j$. \\

        The integrand at this stage is consisted of certain partial derivatives of $H$, derivatives of $\zeta$, and the denominator as a product of linear functions in $s_j, \ldots, s_k$. This is due to the successive substitutions of $s_1, \ldots s_{j-1}$ in terms of $s_j,\ldots, s_k$. Then $$s_1+\cdots + s_{j-1}=-s_j-\cdots -s_k$$ guarantees that for each $j_0\geq j$, there are at least two linear functions (one of them is $s_{j_0}+\alpha_{j_0}$) in the denominator containing nonzero coefficients of $s_{j_0}$. To see this, for each $j_0\geq j$, all substitutions that yields nonzero coefficients of $s_{j_0}$ are from the poles of zeta functions or their derivatives. We write $\kappa_j=\kappa_j(x)$, $\Lm=\Lm(x)$, and $T=T(x)$ (Here, $T$ will be a fixed power of $x$). We apply the method of~\cite[Lemma 20]{BNP}, let $I(x)$ be the integral of the integrand as follows:
        $$
        I(x)=\int\displaylimits_{\kappa_k(x)-iT_k(x)}^{\kappa_k(x)+iT_k(x)}\cdots \int\displaylimits_{\kappa_j(x)-iT_j(x)}^{\kappa_j(x)+iT_j(x)} ( \cdots ) \ ds_j\cdots ds_k
        $$
        Then for sufficiently large $x$, that is, $x\geq x_0$, the contour changes (first horizontal then vertical) to produce $I(2x)$ contribute the errors of $O(\Lm(x)/T(x))$. This gives $|I(2x)-I(x)|=O(\Lm(x)/T(x))$. Thus, $I(2^mx_0)$ forms a Cauchy sequence and it converges to a number $I$ as $m\rightarrow \infty$. Hence, we obtain for sufficiently large $x$ and $2^{m_0}x_0\leq x < 2^{m_0+1}x_0$,
        $$
        |I(x)-I|=O\left(\sum_{m=m_0}^{\infty} \frac{\Lm(2^m x_0)}{T(2^m x_0)}\right)=O\left(\frac{\Lm(x)}{T(x)}\right).
        $$

    $\mathbf{6.}$ The procedure stops when all variables are considered and Step 5 cannot be repeated. During this process, the matrix $L$ goes through row replacements, row interchanges of the first $v+w$ rows, and column interchanges. \\

    We lay out all matrices appearing as results of these row or column operations. There are finitely many such matrices. Let $C\geq 1$ be the maximum of the absolute values of all entries of all matrices in this layout. Let $0<c\leq 1$ be the minimum of the absolute values of all nonzero entries of all matrices in this layout. We require $\kappa<c^2/(2C^2+2)$. This choice of $\kappa$ guarantees that
    $$
    c^2\kappa^j>C^2(\kappa^{j+1}+\kappa^{j+2}+\cdots).
    $$Then for each $j_0\geq j$ and the vertical segment $[\kappa_{j_0}-iT_{j_0}, \kappa_{j_0}+iT_{j_0}]$ of integration, we can decide whether a pole $$s_{j_0}= a_j s_j+\ldots + a_{j_0-1}s_{j_0-1}+a_{j_0+1}s_{j_0+1}+\ldots+a_ks_k$$
    (an obvious modification when $j_0=j, k$) is on the left ($j<j_0$ and $a_j<0$) or right side ($j< j_0$ and $a_j>0$) of the segment from the sign of $a_j$. If $j=j_0$, then the pole does not have $a_js_j$ term and the pole is on the left side of the segment.\\

    As in~\cite[Theorem 12.3]{Ti}, we use Lemma 3.2 and find the contribution of the contour changes to $[-\delta_{j_0}-iT_{j_0}, -\delta_{j_0}+iT_{j_0}]$ or $[\delta_{j_0}-iT_{j_0}, \delta_{j_0}+iT_{j_0}]$ in Step 5 Case 1 and 2. We enlarge $C$ so that $\min\limits_{\substack{{1\leq j\leq k}\\{\alpha_j>0}}} \alpha_j  >\frac1{2kC}$, if necessary. This is to ensure the poles are only from the zeta functions or the linear functions on the denominator. We take $\delta_j = \frac1{2kC}+\epsilon$ for each $j$. The errors due to the integrals over each horizontal side are
    $$
    \ll (T^{\frac v{6k}-1}x^{-\frac1{2kC}} + T^{-1})(xT)^{\epsilon}.
    $$
    The errors due to the integrals over each vertical side are
    $$
    \ll x^{-\frac1{2kC}}T^{\frac1{6k}\max(v-4,0)}(xT)^{\epsilon}.
    $$
    Noting that $\frac{v-6k}{6k}<\frac1{6k}\max(v-4,0)$, we take $T=x^{\frac{v_*}C}$ to balance the error terms $T^{-1}(xT)^{\epsilon}$ and $x^{-\frac1{2kC}}T^{\frac1{6k}\max(v-4,0)}(xT)^{\epsilon}$, then the combined errors become
    $$
    \ll x^{\boldsymbol{\alpha}\cdot {\bf 1}-\frac{v_*}C+\epsilon}.
    $$

    During the process from {\bf 1} to {\bf 6} on these residues, a power of $\log x$ is obtained whenever we take the derivative of the power of $x$ in the integrand. Here, $v+w-\ell$ is the maximal time we take the derivative of powers of $x$.
    \end{proof}
    {\bf Remark. } If the entries of matrix $L$ and all $\alpha_i$'s are integers, then we have the following explicit bounds for $C$ and $c$. Let $\|L\|$ be the maximum of absolute values of entries of $L$. Then
    $$
    C\leq {\|L\|}^k k^{\frac k2} \ \textrm{ and } \ c\geq \|L\|^{-(k-1)} (k-1)^{-\frac {k-1}2}.
    $$
    To see this, suppose that the entry corresponding to $C$ is obtained at $(i_0,j_0)$ position with $i_0\geq i$ and $j_0\geq j$. Up to this stage, $i-1$ pivot positions are determined and $(i-1)$-th pivot position is at $(i-1, j-1)$-th entry. Label the rows of $L_0$ which are selected to have these $i-1$ pivot positions as $R_1, \ldots, R_{i-1}$. Apply a column interchange to move $(i_0,j_0)$-th entry to $(i_0,j)$-th entry and relabel the rows a $R_1,\ldots, R_{i-1}$ after the column interchange. Then apply a row interchange to move it to $(i,j)$-th entry. Label the row containing the pivot at $(i,j)$-th entry as $R_i$. Insert ${\bf e}_i$'s ($i\leq k$) to fill up the missing pivot positions so that the pivot at $(i,j)$-th entry is moved to $(j,j)$-th entry and all $k$ pivot positions are determined. We consider a linear system with an augmented matrix with rows
    $$R_1 + 0 \ {\bf e}_{k+1}, \ldots, R_{i-1}+0 \ {\bf e}_{k+1}, R_i + 1 \ {\bf e}_{k+1},$$
    and ${\bf e}_i$'s used for filling up the pivots. The solution to this system has $s_{j+1}=\cdots=s_k=0$ and $|s_j|=\frac1C$ since the row operations between $R_1,\ldots, R_i$ do not change $1 \ {\bf e}_{k+1}$. Let $B$ be the coefficient matrix of this system. By Cram\'{e}r's rule
    $$
    s_j=\frac{|B'|}{|B|}
    $$
    where $B'$ is $(k-1)\times (k-1)$ minor of $B$ obtained by removing $j$-th row and column.  Since $s_j$ is nonzero and $L$ is integer matrix, we have $|B'|\geq 1$.  By Hadamard's determinant bound we have $|B|\leq {\|L\|}^k k^{\frac k2}$. This yields the upper bound for $C$. The lower bound for $c$ is obtained by the similar method with $|B|\geq 1$ and $|B'|\leq \|L\|^{k-1} (k-1)^{\frac{k-1}2}$. In case $L$ is a zero-one matrix, Hadamard's determinant bound is stronger. We have
    $$
    C\leq \frac{(k+1)^{\frac{k+1}2}}{2^k} \ \textrm{ and } \ c\geq \frac{2^{k-1}}{k^{\frac k2}}.
    $$

    We are able to strengthen Theorem 2.2 to the full asymptotic with a power saving error term.
    \begin{theorem}
    We have
    $$
    \sum_{\substack{{a_1,\ldots, a_v\leq x}\\{G\textrm{-wise coprime}}\\{\forall_{i\leq k}, \ \prod_{j\in A_i}a_j \leq x}}} \frac1{a_1\cdots a_v}=P_v(\log x)+O(x^{-\theta+\epsilon})
    $$
    where $P_v$ has degree $v$ with the leading coefficient $\rho(G)\mathrm{vol}(D)$ in view of Theorem 2.2 and
    $$
    \theta=\frac{2^k}{(k+1)^{\frac{k+1}2}}\cdot \frac3{6k+\max(v-4,0)}.
    $$
    \end{theorem}
    \begin{proof}
    This follows from Lemma 3.1 and Proposition 3.2. The arithmetic function
    $$
    f(n_1,\ldots, n_k)=\sum_{\substack{{\forall_r, a_r\in \N}\\{G\textrm{-wise coprime}}\\{\forall_{i\leq k}, \ \prod_{j\in A_i} a_j = n_i}}}\frac1{a_1\cdots a_v}
    $$
    in (13) satisfies the conditions of Proposition 3.2. We have $\boldsymbol{\alpha}={\bf 0}$ so that $|J(\boldsymbol{\alpha})|=k$ and $w=\ell=k$. Hence, the result follows.
    \end{proof}
    \subsection{Proof of Theorem 1.1}
    (6) and (7) follow from the analysis of Section 3.2 and Theorem 3.1 with $v=2^k-1$ and $v=2^k-2$ respectively. Thus, for $k\geq 3$, (6), (7), and (8) hold with
    $$
    \theta^{(1)}_k=\theta^{(3)}_k=\frac{2^k}{(k+1)^{\frac{k+1}2}}\cdot \frac 3{2^k+6k-5}, \ \ \theta^{(2)}_k=\frac{2^k}{(k+1)^{\frac{k+1}2}}\cdot \frac 3{2^k+6k-6}.
    $$
    We have $w=\ell=k$ in these two cases (6) and (7). The main term of (8) is obtained by $\boldsymbol{\alpha}\cdot {\bf 1}=k$, $v=2^k-1$, $w=0$, $\ell=k$, and Proposition 3.2. \\

    By Section 3.2, the leading coefficients of $P^{(1)}_{2^k-1}$ and $P^{(2)}_{2^k-2}$ are $c_k=\rho(G_k)\mathrm{vol}(D_k)$ and $c_k^{(2)}=(2^k-1)c_k$ respectively. Computations of $\rho(G_k)$ and $\mathrm{vol}(D_k)$ are continued on Section 4. The leading coefficient of $P^{(3)}_{2^k-k-1}$ is positive in view of~\cite[Theorem 1.3]{K} or~\cite[Theorem 2]{dlB}. We have $\ell=k$ for all three cases (6), (7), and (8). \\

    \section{Computations of $\rho(G_k)$ and $\mathrm{vol}(D_k)$}
    \subsection{The number $\rho(G_k)$}
    In~\cite[(2.5)]{HLT}, it was proved that
    $$
    \sum_{n\leq x} \frac{\alpha_k(n)}n \leq S_k(x)
    $$
    where $\alpha_k(n)=\sum_{\substack{{n_1,\ldots, n_k\in\N}\\{[n_1,\ldots,n_k]=n}}} 1$. Using $\alpha_k$ is multiplicative and $\alpha_k(p)=2^k-1$ for each prime $p$, they applied~\cite[Theorem 2.2]{HLT} to prove
    $$
    \sum_{n\leq x}\frac{\alpha_k(n)}n=\frac{C_k}{(2^k-1)!}\log^{2^k-1}x+O(\log^{2^k-2}x)
    $$
    where
    $$
    C_k=\prod_{p\in\mathcal{P}}\left(1-\frac1p\right)^{2^k-1} \sum_{\nu=0}^{\infty} \frac{(\nu+1)^k-\nu^k}{p^{\nu}}.
    $$
    In view of Theorem 2.2, we have
    $$
    \sum_{n\leq x}\frac{\alpha_k(n)}n=\sum_{\substack{{a_1,\ldots, a_{v_k}\leq x}\\{G_k\textrm{-wise coprime}}\\{a_1\cdots a_{v_k}\leq x}}}\frac1{a_1\cdots a_{v_k}}=\rho(G_k)\textrm{vol}(T_k)\log^{2^k-1}x+O(\log^{2^k-2}x)
    $$
    where
    $$
    T_k=\{(u_1,\ldots, u_{v_k})\in [0,\infty)^{v_k} \ | \ u_1+\cdots +u_{v_k}\leq 1\}.
    $$
    Since $\textrm{vol}(T_k)=\frac1{(2^k-1)!}$, we have
    $$
    \rho(G_k)=C_k=\prod_{p\in\mathcal{P}}\left(1-\frac1p\right)^{2^k-1} \sum_{\nu=0}^{\infty} \frac{(\nu+1)^k-\nu^k}{p^{\nu}}.
    $$
    Thus, we are able to evaluate $\rho(G_k)$ without knowing precise information about $G_k$.

    We have three expressions for $\rho(G_k)$ from~\cite{RH},~\cite{H2}, and~\cite{HLT}. These suggest the identities for $|x|<1$,
    $$
    \sum_{F\subseteq E_k} (-1)^{|F|}x^{v(F)}=\sum_{m=0}^{v_k} i_m(G_k)\left(1-x\right)^{v_k-m}x^m=(1-x)^{2^k-1}\sum_{\nu=0}^{\infty} ( (\nu+1)^k-\nu^k)x^{\nu}.
    $$
    Analyzing the proofs in these papers, the values at $1/p$ for each prime $p$ are identical. Thus, we see that the above identities are indeed true.

    We may recover some information about $G_k$ by the last expression. For example, we obtain the number of edges in $G_k$ by equating the coefficients of $x^2$, thus
    $$
    -|E_k|=(3^k-2^k) - (2^k-1)^2 + \binom{2^k-1}2=3^k-2^{k-1}(2^k+1)
    $$
    so that $|E_2|=1$, $|E_3|=9$, and $|E_4|=55$. Also, we have
    $$
    \sum_{\nu=0}^{\infty} \nu^k x^{\nu} = \frac1{1-x} \sum_{m=1}^k S(k,m) m !\left(\frac x{1-x}\right)^m
    $$
    where $S(k,m)$ is the Stirling number of the second kind. Then
    $$
    \sum_{\nu=0}^{\infty} ( (\nu+1)^k-\nu^k)x^{\nu}=\frac1x\sum_{m=1}^k S(k,m) m !\left(\frac x{1-x}\right)^m.
    $$
    By $\frac1x= 1+ \left(\frac x{1-x}\right)^{-1}$, we obtain $i_m(G_k)$ in terms of the Stirling numbers. To see this, let $t=\frac x{1-x}$, we have
    $$
    \sum_{m=0}^{v_k} i_m(G_k) \ t^m=\left(1+\frac1t\right)\sum_{m=1}^k S(k,m) m!  \ t^m.
    $$
    Hence,
    $$
    i_m(G_k)=S(k,m)m! + S(k,m+1)(m+1)!.
    $$
    Since the above vanishes when $m>k$, we have
    $$
    \sum_{F\subseteq E_k} (-1)^{|F|}x^{v(F)}=\sum_{m=0}^k \big(S(k,m)m! + S(k,m+1)(m+1)!\big)\left(1-x\right)^{v_k-m}x^m.
    $$
    Therefore, we are able to write $f_{G_k}(p^a)$ in terms of the Stirling numbers and binomial coefficients. The following are the polynomials used in the Euler product of $\rho(G_k)$ obtained after expanding the right side.
    \begin{align*}
    Q_{G_2}(x)&=1-x^2,\\
    Q_{G_3}(x)&=1-9x^2+16x^3-9x^4+x^6,\\
    Q_{G_4}(x)&=1-55x^2+320x^3-891x^4+1408x^5-1155x^6+1155x^8-1408x^9\\
    &+891x^{10}-320x^{11}+55x^{12}-x^{14}.
    \end{align*}
    \subsection{The polytope $D_k^*$}
    As noted in Section 3.1, the volume of polytope defined by a finite set of inequalities can be computed by a SageMath worksheet. Since $\mathrm{vol}(D_k)=\mathrm{vol}(D_k^*)/(2^k-1)$, we use the following code to generate a SageMath worksheet for $D_k^*$.
    \begin{lstlisting}[title=A python code generating a SageMath worksheet for $D_k^*$]
file2 = open('Polytope.txt', 'w+')
k=input()
s=[[0 for y in range(2**k-1)] for x in range(k+2**k-2)]
for i in range(k): #Giving hyperbolic constraints.
    s[i][0]=1
    for j in range(2**k-1):
        if j&(1<<i):
            s[i][j]=-1  #When i-th bit from the right is 1, assign -1.
for z in range(k,k+2**k-2): #Restricting variables to be nonnegative.
    s[z][z-k+1]=1
file2.write('P=Polyhedron(ieqs=')
file2.write(str(s))
file2.write(')')
file2.write('\n')
file2.write('P.volume()')
file2.close()
\end{lstlisting}
    The output in case $k=3$ is
    \begin{lstlisting}
P=Polyhedron(ieqs=[[1, -1, 0, -1, 0, -1, 0],
[1, 0, -1, -1, 0, 0, -1],
[1, 0, 0, 0, -1, -1, -1],
[0, 1, 0, 0, 0, 0, 0],
[0, 0, 1, 0, 0, 0, 0],
[0, 0, 0, 1, 0, 0, 0],
[0, 0, 0, 0, 1, 0, 0],
[0, 0, 0, 0, 0, 1, 0],
[0, 0, 0, 0, 0, 0, 1]])
P.volume()
    \end{lstlisting}
    The output in case $k=4$ is
    \begin{lstlisting}
P=Polyhedron(ieqs=[
[1, -1, 0, -1, 0, -1, 0, -1, 0, -1, 0, -1, 0, -1, 0],
[1, 0, -1, -1, 0, 0, -1, -1, 0, 0, -1, -1, 0, 0, -1],
[1, 0, 0, 0, -1, -1, -1, -1, 0, 0, 0, 0, -1, -1, -1],
[1, 0, 0, 0, 0, 0, 0, 0, -1, -1, -1, -1, -1, -1, -1],
[0, 1, 0, 0, 0, 0, 0, 0, 0, 0, 0, 0, 0, 0, 0],
[0, 0, 1, 0, 0, 0, 0, 0, 0, 0, 0, 0, 0, 0, 0],
[0, 0, 0, 1, 0, 0, 0, 0, 0, 0, 0, 0, 0, 0, 0],
[0, 0, 0, 0, 1, 0, 0, 0, 0, 0, 0, 0, 0, 0, 0],
[0, 0, 0, 0, 0, 1, 0, 0, 0, 0, 0, 0, 0, 0, 0],
[0, 0, 0, 0, 0, 0, 1, 0, 0, 0, 0, 0, 0, 0, 0],
[0, 0, 0, 0, 0, 0, 0, 1, 0, 0, 0, 0, 0, 0, 0],
[0, 0, 0, 0, 0, 0, 0, 0, 1, 0, 0, 0, 0, 0, 0],
[0, 0, 0, 0, 0, 0, 0, 0, 0, 1, 0, 0, 0, 0, 0],
[0, 0, 0, 0, 0, 0, 0, 0, 0, 0, 1, 0, 0, 0, 0],
[0, 0, 0, 0, 0, 0, 0, 0, 0, 0, 0, 1, 0, 0, 0],
[0, 0, 0, 0, 0, 0, 0, 0, 0, 0, 0, 0, 1, 0, 0],
[0, 0, 0, 0, 0, 0, 0, 0, 0, 0, 0, 0, 0, 1, 0],
[0, 0, 0, 0, 0, 0, 0, 0, 0, 0, 0, 0, 0, 0, 1]])
P.volume()
    \end{lstlisting}
    Running these on SageMath, we obtain $\mathrm{vol}(D_3^*)=11/480$ and $\mathrm{vol}(D_4^*)=739/25830604800$. Thus, $\mathrm{vol}(D_3)=11/3360$ and $\mathrm{vol}(D_4)=739/387459072000$.
    \subsection{The polytope $D_k^{**}$}
    Recall that
    $$
    D^{**}_k=\left\{(t_i)_{\substack{{i\leq {v_k}}\\{i\neq 2^t}}}\in [0,\infty)^{2^k-k-1} \ \bigg\vert  \forall_{i\leq k}, \ \sum_{j\in A_{k,i}-\{2^{i-1}\}}t_j \leq 1 \right\}.
    $$
    The volume of $D^{**}$ is $1/(2^k-k-1)$ times the volume of
    $$
    D^{***}_k=\left\{(t_i)_{\substack{{i\leq {v_k-1}}\\{i\neq 2^t}}}\in [0,\infty)^{2^k-k-2} \ \bigg\vert  \forall_{i\leq k}, \ \sum_{j\in A_{k,i}-\{2^{i-1}, v_k\}}t_j \leq 1 \right\}.
    $$
    We use the following code to generate a SageMath worksheet for $D_k^{***}$.
    \begin{lstlisting}[title=A python code generating a SageMath worksheet for $D_k^{***}$]
file2 = open('Polytope.txt', 'w+')
k=input()
s=[[0 for y in range(2**k-1)] for x in range(2*k+2**k-2)]
for i in range(k): #Giving hyperbolic constraints.
    s[i][0]=1
    for j in range(2**k-1):
        if j&(1<<i):
            s[i][j]=-1
            #When i-th bit from the right is 1, assign -1.
for z in range(k,k+2**k-2): #Restricting variables to be nonnegative.
    s[z][z-k+1]=1
for u in s:
    for v in range(k):
        del u[2**v-v] #Delete positions 2^j.
rows = [i for i in s if set(i)!={0}]
file2.write('P=Polyhedron(ieqs=')
file2.write(str(rows))
file2.write(')')
file2.write('\n')
file2.write('P.volume()')
file2.close()
    \end{lstlisting}
    The output in case $k=3$ is
 \begin{lstlisting}
P=Polyhedron(ieqs=[[1, -1, -1, 0], [1, -1, 0, -1],
[1, 0, -1, -1], [0, 1, 0, 0], [0, 0, 1, 0], [0, 0, 0, 1]])
P.volume()
\end{lstlisting}
    The output in case $k=4$ is
\begin{lstlisting}
P=Polyhedron(ieqs=[
[1, -1, -1, 0, -1, -1, 0, -1, 0, -1, 0],
[1, -1, 0, -1, -1, 0, -1, -1, 0, 0, -1],
[1, 0, -1, -1, -1, 0, 0, 0, -1, -1, -1],
[1, 0, 0, 0, 0, -1, -1, -1, -1, -1, -1],
[0, 1, 0, 0, 0, 0, 0, 0, 0, 0, 0],
[0, 0, 1, 0, 0, 0, 0, 0, 0, 0, 0],
[0, 0, 0, 1, 0, 0, 0, 0, 0, 0, 0],
[0, 0, 0, 0, 1, 0, 0, 0, 0, 0, 0],
[0, 0, 0, 0, 0, 1, 0, 0, 0, 0, 0],
[0, 0, 0, 0, 0, 0, 1, 0, 0, 0, 0],
[0, 0, 0, 0, 0, 0, 0, 1, 0, 0, 0],
[0, 0, 0, 0, 0, 0, 0, 0, 1, 0, 0],
[0, 0, 0, 0, 0, 0, 0, 0, 0, 1, 0],
[0, 0, 0, 0, 0, 0, 0, 0, 0, 0, 1]])
P.volume()
\end{lstlisting}
Running these on SageMath, we obtain $\mathrm{vol}(D_3^{**})=\mathrm{vol}(D_3^{***})/4=1/16$ and $\mathrm{vol}(D_4^{**})=\mathrm{vol}(D_4^{***})/11=299/479001600$.
       \flushleft

\end{document}